\documentclass[11pt]{amsart}
\usepackage{amssymb}
\usepackage{amsmath}	
\usepackage{bm}

\usepackage{verbatim}

\calclayout

\usepackage{graphicx}
\usepackage{epstopdf}

\newcommand\inlinegraphic[2][{scale=1.0}]{
\begin{array}
  {c} \includegraphics[#1]{./minimal-graphics/#2}
\end{array}
}

\usepackage{xcolor}

\usepackage{enumitem}

\usepackage[plainpages=false, pdfpagelabels,  colorlinks=true, pdfstartview=FitV, linkcolor=blue, citecolor=blue, urlcolor=blue]{hyperref}

\numberwithin{equation}{section}
\numberwithin{figure}{section}

\newcommand\la{\lambda}
\newcommand\La{\Lambda}

\newcommand\ga{\gamma}

\newcommand\N{\mathbb{N}}
\newcommand\Z{\mathbb{Z}}
\newcommand\Q{\mathbb{Q}}

\DeclareMathOperator{\tr}{tr}
\DeclareMathOperator{\End}{End}
\DeclareMathOperator{\Ad}{Ad}
\DeclareMathOperator{\alg}{alg}
\DeclareMathOperator{\spn}{span}
\DeclareMathOperator{\res}{Res}
\DeclareMathOperator{\cl}{cl}
\DeclareMathOperator{\character}{char}
\DeclareMathOperator{\ord}{ord}

\newcommand\inv{^{-1}}
\newcommand\powerpm{^{\pm 1}}

\newcommand\qBr{\mathfrak B}  
\newcommand\Br{B}    
\newcommand\BMW{\mathfrak D}

\newcommand\ex[1]{e_{(#1)}}
\newcommand\barex[1]{{\bar e}_{(#1)}}
\newcommand\exprime[1]{e'_{(#1)}}

\newcommand\mfs{\mathfrak s}
\newcommand\mft{\mathfrak t}

\newcommand\eps{\varepsilon}

\newcommand{\cell}[2]{\Delta_{#1}^{#2}}
\newcommand{\Res}[3]{\res^{#1}_{#2}(#3)}

\newcommand{\qinteger}[2]{[#1]_{#2}}
\newcommand\qint{\qinteger}

\theoremstyle{plain}
\newtheorem{theorem}{Theorem}[section]

\theoremstyle{plain}
\newtheorem{proposition}[theorem]{Proposition}

\theoremstyle{plain}
\newtheorem{corollary}[theorem]{Corollary}

\theoremstyle{plain}
\newtheorem{lemma}[theorem]{Lemma}

\theoremstyle{definition}
\newtheorem{definition}[theorem]{Definition}
\theoremstyle{definition}
\newtheorem{example}[theorem]{Example}

\theoremstyle{definition}
\newtheorem{remark}[theorem]{Remark}

\theoremstyle{definition}

\theoremstyle{definition}
\newtheorem{notation}[theorem]{Notation}

\theoremstyle{plain}

\title{Semisimplicity Criteria for algebras with a Jones Basic Construction}

\author{Frederick M. Goodman}
\address{Department of Mathematics\\ University of Iowa\\ Iowa
City, Iowa}
\email{ goodman@uiowa.edu}

\author{Hans Wenzl}
\address{Department of Mathematics\\ University of California, San Diego\\San Diego, California}
\email{ hwenzl@ucsd.edu}

\subjclass[2000]{}

\begin{document}
\begin{abstract}
	We prove a semisimplicity criterion for a large class of algebras by a new method. This can be applied to Brauer, BMW, and $q$-Brauer
algebras.
\end{abstract}
\maketitle
\medskip

\section{Introduction}  We develop a new method for determining semisimplicity of algebras $A_n$ that
generically are obtained by repeated Jones basic constructions from a simpler
family of algebras $Q_n$. Such algebras include the Brauer algebras,
Birman-Wenzl-Murakami (BMW) algebras, and the $q$-Brauer algebras.
Our method is valid over a field of arbitrary characteristic.

Let us first describe our results in the context of the examples mentioned
above. These algebras appear as centralizer algebras of certain tensor
categories (Brauer algebras, BMW algebras) or module categories ($q$-Brauer
algebras) that carry natural dimension functions. The dimension functions give
rise to a special trace functional, the Markov trace; let $n_0$ be the smallest
value of $n$ such that the Markov trace is degenerate on $A_n$ -- that is,
there exists $0\neq x \in A_n$ such that $\tr(x y) = 0$ for all $y \in A_n$.
Moreover, let $n_1 +1 $ be the smallest value of $n$ such that the quotient
algebra $Q_n$ fails to be semisimple. (Put $n_0 = \infty$ in case the Markov
trace is always non-degenerate, and $n_1 = \infty$ in case $Q_n$ is always
semisimple.) By arguments presented by the second author for the Brauer
algebras ~\cite{Wenzl-Brauer}, and later for the BMW algebras ~\cite{Wenzl-BCD}
and $q$-Brauer algebras ~\cite{Wenzl-qBr1}, ~\cite{Wenzl-qBr2}, $A_n$ is
semisimple for all $n \le \min\{n_0, n_1\}$. Explicit bounds are obtained using
formulas for the dimension functions; these are El Samra-King polynomials for
the Brauer algebras and two different $q$-deformations of them for the BMW
algebras and the $q$-Brauer algebras, see Section \ref{explicit:sec} for
details.

To prove that these bounds are sharp, we exploit the fact that each of these
algebras is cellular ~\cite{Graham-Lehrer}. We show that the cellular structure
of $A_{k+2s}$ is related to the cellular structure of $A_k$ in a particular way
related to the Jones basic construction. Using this, we show that $A_n$ is
non-semisimple for $n > \min\{n_0, n_1\}$. This allows us to obtain explicit
semisimplicity criteria for these families of algebras, recovering results of
Rui et al.\ (\cite{Rui-Brauer-semisimplicity.1, Rui-Brauer-semisimplicity.2,
	Rui-Brauer-discriminants, Rui-BMW-Gram-determinants, rui2022jm}). We note that
for the Brauer and BMW algebras, Rui et al.\ have also determined for each cell
module whether the cell module is simple, which is a stronger result.

We prove our results for a general class of algebras determined by axioms
listed in Section \ref{axioms}. We then prove that the algebras $A_n$ in this
class are semisimple if and only if $n \le \min\{n_0, n_1\}$, in Theorems
\ref{first semisimplicity theorem - axiomatic version} and \ref{second
	semisimplicity theorem - axiomatic version}. We review properties of Brauer
algebras, BMW algebras and $q$-Brauer algebras in Section \ref{examples:sec}
and show that they satisfy our axioms. We then explicitly calculate $\min\{n_0,
	n_1\}$ for these examples in Section \ref{explicit:sec}, for all possible
choices of parameters.

The theoretical novelty of our method is that it plays the Jones basic
construction against the cellular structure of the algebras $A_n$ to obtain
criteria for semisimplicity involving the weights of the Markov trace. The
method does not require computation of seminormal representations or Gram
determinants. The practical advantage is that applying our conditions requires
only computing vanishing conditions for the weights of the Markov trace, rather
than for the Gram determinants, and that turns out to be relatively easy. A
limitation of this approach is that we cannot obtain all the information
encoded in the Gram determinants. We expect that our approach would also be
useful for other algebras related to the Jones basic construction, such as
e.g.\ walled Brauer algebras and cyclotomic BMW algebras.

\section{Preliminaries}
In this section, we review background information on embeddings of semisimple
algebras and on cellular algebras.

\subsection{Branching diagrams }  \label{section: branching diagrams}

Suppose $\varphi: A \to B$ is a unital homomorphism of finite dimensional split
semisimple algebras over a field $F$. Let $A(i)$, $i \in I$, be the minimal
two-sided ideals of $A$ and $B(j)$, $j \in J$, the minimal two-sided ideals of
$B$. We associate a $J \times I$ {\em inclusion matrix} $\omega$ to the
homomorphism $\varphi: A \to B$, as follows. Let $W_j$ be a simple
$B(j)$--module. Then $W_j$ becomes an $A$--module via the homomorphism, and
$\omega(j, i)$ is the multiplicity of a simple $A(i)$--module in the
decomposition of $W_j$ as an $A$--module. An equivalent characterization of the
inclusion matrix is the following. Let $q_i$ be a minimal idempotent in $A(i)$
and let $z_j$ be the minimal central idempotent in $B$ such that $z_j B =
	B(j)$. Then $\varphi(q_i) z_j$ is the sum of $\omega(j, i)$ minimal idempotents
in $B(j)$. We are mostly interested in injective homomorphisms or inclusions $A
	\subseteq B$, but at one point we need the more general setup.

It is convenient to encode an inclusion matrix by a bipartite graph, called the
	{\em branching diagram}; the branching diagram has vertices labeled by $I$
arranged on one horizontal line, vertices labeled by $J$ arranged along a
second (higher) horizontal line, and $\omega(j, i)$ edges connecting $j \in J$
to $i \in I$.

If $A_1 \to A_2 \to A_3 \cdots$ is a (finite or infinite) sequence of
homomorphisms of finite dimensional split semisimple algebras over $F$, then
the branching diagram for the sequence is obtained by stacking the branching
diagrams for each inclusion, with the upper vertices of the diagram for $A_i
	\to A_{i+1}$ being identified with the lower vertices of the diagram for
$A_{i+1} \to A_{i+2}$.

For our purposes, it suffices to restrict our attention to the case that the
entries in each inclusion matrix are all $0$ or $1$; thus in the branching
diagram there are no multiple edges between vertices. We say that the
homomorphisms (and the branching diagram) are {\em multiplicity free}. We can
also assume that $A_0 = F$, so that there is a unique vertex $\emptyset$ at
level zero of the branching diagram, which we call the {\em root} of the
branching diagram.

The branching diagram of a sequence of multiplicity free homomorphisms is an
abstract branching diagram in the following sense:

\begin{definition}
	\label{def of branching}  A (finite or infinite, multiplicity free) abstract branching diagram $\Gamma$  is a  graph  with vertex set
	$ \coprod_{i \ge 0} \Gamma_i$, with the following properties
	\begin{enumerate}
		\item $\Gamma_0$ is a singleton and $\Gamma_i$ is finite for all $i$.
		\item Two vertices $v \in \Gamma_i$ and $w \in \Gamma_j$ are adjacent only if $|i -
			      j|=1$. Adjacent vertices are connected by a unique edge.
		\item  If $i \ge 1$ and $v \in \Gamma_i$, then $v$ is adjacent to at least one vertex
		      in $\Gamma_{i-1}$.
	\end{enumerate}
	We write $v \nearrow w$   if $v \in \Gamma_i$ is adjacent to $w \in \Gamma_{i+1}$.     If $v \in \Gamma_i$, we write $|v|  = i$.
\end{definition}

Given a multiplicity-free branching diagram $\Gamma$, we define a new branching
diagram $\tilde \Gamma$ as follows. The vertex set $\tilde \Gamma_n$ at level
$n$ is the set of pairs $(v, n)$ where $v \in \Gamma_k$ for some $k \le n$ with
$n - k$ even. The branching rule is $(v, n) \nearrow (w, n+1)$ if and only if
$|w| = |v| + 1$ and $v \nearrow w$ in $\Gamma$, or $|w| = |v| - 1$ and $w
	\nearrow v$ in $\Gamma$. We say that $\tilde \Gamma$ is the branching diagram
obtained {\em by reflections} from $\Gamma$.

\begin{example}
	(Young's lattice) \label{Young's lattice}
	For $n \ge 0$,  let $\Lambda_n$ denote the set of Young diagrams (or integer partitions) of size $n$.    {\em Young's lattice} is the disjoint
union of $\Lambda_n$ for $n \ge 0$, with the following branching rule:
	for $\lambda \in \Lambda_n$ and $\mu \in \Lambda_{n+1}$,  $\lambda \to \mu$ if $\mu$ is obtained from $\lambda$ by adding one box.

	Over a field of characteristic zero, when the parameter $q$ is not a root of
	unity, the Hecke algebras $H_n(q)$ (see Section \ref{subsection Hecke algebras}) are split semisimple for all $n$, the simple
	modules of $H_n(q)$ are labelled by $\Lambda_n$, and the branching rule for the
	inclusion $H_n \subseteq H_{n+1}$ is as described above.

	Each $\Lambda_n$ has a partial order (dominance order) that plays a role in the
	representation theory of the Hecke algebras, in particular in the definition of
	a cellular structure on $H_n(q)$. The definition is as follows: $\lambda \unrhd
		\mu$ in $\Lambda_n$ if for all $k$, $$\lambda_1 + \lambda_2 + \cdots +
		\lambda_k \ge \mu_1 + \mu_2 + \cdots + \mu_k.$$
\end{example}

\begin{example}
	\label{lattice for generic Brauer algebras}
	Let $\tilde \Lambda$ denote the branching diagram obtained by reflections from Young's lattice.
	$\tilde \Lambda$ is the branching diagram for the sequence of Brauer algebras in the generic semisimple case:  over a field of characteristic
zero and when the parameter $\delta$ is not an integer, the Brauer algebras $\Br_n(\delta)$ are split semisimple, the simple modules of
$\Br_n(\delta)$ are labelled by $\tilde \Lambda_n$, and the branching rule for $\Br_n \subseteq \Br_{n+1}$ is that obtained by reflections from
Young's lattice.

	Each $\tilde \Lambda_n$ has a partial order that plays a role in the
	representation theory of the Brauer algebras. The definition is $(\lambda, n)
		\unrhd (\mu, n)$ if $|\lambda| < |\mu|$, or if $|\lambda| = |\mu|$ and $\lambda
		\unrhd \mu$ in dominance order for Young diagrams.
\end{example}

\subsection{Cellular Algebras}  \label{section cellular algebras}
We recall the definition of {\em cellularity} from ~\cite{Graham-Lehrer}; see
also ~\cite{Mathas-book}. We use a slightly weaker version of the definition,
which was introduced in ~\cite{MR2510050} . In the following, an {\em algebra
		involution} $*$ on an algebra $A$ is an algebra anti-automorphism such that
$(a^*)^* = a$ for all $a \in A$.

\begin{definition}
	\label{gl cell}  Let $R$ be an integral domain and $A$ a unital $R$--algebra.  A {\em cell datum} for $A$ consists of  an algebra involution
$*$ of $A$; a partially ordered set $(\Lambda, \ge)$ and
	for each $\la \in \Lambda$  a set $\mathcal T(\lambda)$;  and   a subset $
		\mathcal C = \{ c_{s, t}^\la :  \la \in \Lambda \text{ and }  s, t \in \mathcal T(\la)\} \subseteq A$;
	with the following properties:
	\begin{enumerate}
		\item  $\mathcal C$ is an $R$--basis of $A$.
		\item   \label{mult rule} For each $\la \in \Lambda$,  let $A^{> \la}$  be the span of the  $c_{s, t}^\mu$  with
		      $\mu > \la$.   Given $\la \in \Lambda$,  $s \in \mathcal T(\la)$, and $a \in A$,   there exist coefficients
		      $r_v^s( a) \in R$ such that for all $t \in \mathcal T(\la)$:
		      $$
			      a c_{s, t}^\la  \equiv \sum_v r_v^s(a)  c_{v, t}^\la  \mod  A^{> \la}.
		      $$
		\item  $(c_{s, t}^\la)^* \equiv c_{t, s}^\la   \mod  A^{> \la}$ for all $\la\in \Lambda$ and $s, t \in \mathcal T(\lambda)$.

	\end{enumerate}
	$A$ is said to be a {\em cellular algebra} if it has a  cell datum.
\end{definition}

For brevity, we will write that $(\mathcal C, \La)$ is a cellular basis of $A$.

We recall some basic structures related to cellularity, see
~\cite{Graham-Lehrer}. First we note that for any $R$--algebra $A$ with involution
$*$, the involution induces a functor interchanging left and right modules, by
composition with $*$. (If $\Delta$ is a left $A$--module,  then $\Delta^*$  is
the same as $\Delta$ as an $R$--module, and has the right $A$--module structure
defined by $ x a = a^* x$ for $x \in \Delta$ and $a \in A$.)
Given $\la\in\La$, let $A^{\ge \la}$ denote the span of
the $c_{s,t}^{\mu}$ with $\mu \geq \la$. It follows that both $A^{\ge \la}$ and
$A^{> \la}$ (defined above) are $*$--invariant two sided ideals of $A$. The
	{\em left cell module} $\cell {\la} {}$ is defined as follows: as an
$R$--module, $\cell {\la}{}$ is free with basis \break $\{c_s^\la$ : $s \in
	\mathcal T(\la)\}$; for each $a \in A$, the action of $a$ on $\cell {\la}{}$ is
defined by $ ac_s^\la=\sum_v r_v^s(a) c_v^\la$ where $r_v^s(a)$ is as in
Definition \ref{gl cell} (\ref{mult rule}). For all $s,t \in \mathcal T(\la)$,
we have a canonical $A-A$--bimodule isomorphism $\alpha : A^{\ge \la}/A^{> \la}
	\rightarrow \cell {\la}{} \otimes_R (\cell {\la}{})^*$ defined by
$\alpha(c_{s,t}^{\la}+A^{> \la})=c_s^\la \otimes_R c_t^\la$.

\subsubsection{Semisimple cellular algebras}

A cellular algebra $A$ over a field is split semisimple if it is semisimple;
and it is semisimple if and only if each cell module is simple
(\cite{Graham-Lehrer}, Theorem 3.8). In this case, the cell modules coincide
with the simple modules, and are mutually non-isomorphic. For each $\la \in
	\La$, there is a minimal central idempotent $z_\la$ such that $z_\la A$ is a
minimal two sided ideal isomorphic to $\End(\cell {\la}{})$.

\begin{lemma}
	\label{semisimple cellular algebras} Let $A$ be a semisimple cellular algebra over a field, with cellular basis $\{c_{\mathfrak s, \mathfrak
t}^\la : \la \in \La,  \mfs, \mft \in \mathcal T(\la)\}$.
	\begin{enumerate}
		[label = {\normalfont (\arabic*)}]
		\item  For $x \in A^{> \la}$, $x z_\la = 0$. \item The map $x + A^{> \la}\to x z_\la$
		      is an $A$--$A$ bimodule isomorphism from $A^{\ge \la}/A^{>\la}$ to $z_\la A$.
		\item  $z_\la \in A^{\ge \la} \setminus A^{> \la}$.
		\item  If $p$ is a minimal idempotent in $z_\la A$, then $p \in A^{\ge \la} \setminus
			      A^{> \la}$.
	\end{enumerate}
\end{lemma}

\begin{proof}
	We have
	$$
		z_\lambda \in z_\lambda A \cong  \End(\cell {\la}{})  \cong  \dim(\cell {\la}{})  \cell {\la}{},
	$$
	as $A$-modules.
	Since $A^{> \la}$ acts as zero on $\cell {\la}{}$,  it follows that $x z_\la = 0$ for $x \in A^{> \la}$.
	Therefore, the map   $x + A^{> \la}\to x z_\la$  from $A^{\ge \la}/A^{>\la}$ to $z_\la A$  is well defined.
	A model for the (simple) cell module $\cell {\la}{}$ is $\spn\{c_{\mathfrak s, \mathfrak t}^\la + A^{> \la} : s \in \mathcal T(\la)\}$ for any
fixed $\mft \in \mathcal T(\la)$.  Since $z_\la$ acts as the identity on $\cell {\la}{}$, we have $z_\la  c_{\mfs, \mft}^\la \equiv c_{\mfs,
\mft}^\la  \mod A^{>\la}$. Therefore, $x + A^{> \la}\to x z_\la$ is injective from $A^{\ge \la}/A^{>\la}$ to $z_\la A$.   Since both have dimension
$\dim(\cell {\la}{})^2$, the map is an isomorphism.   This proves parts (1) and (2).   By surjectivity in part (2),
	$$
		z_\la \in z_\la A = \spn\{c_{\mathfrak s, \mathfrak t}^\la z_\la : \mfs, \mft \in \mathcal T(\la)\} \subseteq A^{\ge \la}.
	$$
	Since $z_\la = z_\la^2$,  it follows that $z_\la \not \in A^{> \la}$,  by part (1).    This proves part (3).  Part (4) also follows, since $p =
p z_\la$.
\end{proof}

\subsubsection{Generic ground rings}   The examples of cellular algebras of interest to us are actually families $A^S$
of algebras defined over various integral ground rings $S$, possibly containing
distinguished elements (parameters) that enter into the definition of the
algebras. In these examples, there is a ``generic ground ring'' $R$ for $A$
with the following property: Any suitable integral ground ring $S$ for $A$ is
an $R$-module, and $A^S$ is the specialization of $A^R$, that is $A^S \cong
	A^{R} \otimes_{R} S$. Likewise, the cell modules of $A^S$ are specializations
of those of $A^{R}$, that is $\cell \la {S} \cong \cell \la {{R}} \otimes_{R}
	S$. In examples of interest, $R$ has characteristic zero, and if $F$ denotes
the field of fractions of $R$, then $A^{F}$ is split semisimple.

A prototypical example is the Brauer algebra $B_n = B_n(S; \delta)$, which is
an algebra with a basis of diagrams that can be defined over any integral
ground ring $S$ with a distinguished element $\delta$. The product of any two
basis diagrams is a power of $\delta$ times a third diagram. For the Brauer
algebras and similar examples, we will insist that the ``loop parameter"
$\delta$ is invertible in $S$. Thus the generic ground ring for the Brauer
algebras is $R= \Z[\bm \delta\powerpm]$, where $\bm \delta$ is an
indeterminate.

\subsubsection{Coherent towers and branching diagrams}  \label{Coherent towers}  Let $A$ be a cellular algebra. A {\em cell filtration} of an
$A$-module is a
filtration by $A$-submodules such that the subquotients are isomorphic to cell
modules. Now consider an increasing sequence $(A_i)$ of cellular algebras (with
the same identity).
One says that the sequence is {\em restriction coherent} if for each $i$, and
for each cell module $\Delta$ of $A_{i+1}$, the restricted module $\Res
	{A_{i+1}} {A_i} \Delta$ has a cell filtration.

\begin{lemma}
	\label{lemma: permanent multiplicities}   {\normalfont  (Permanence of multiplicities) } Let $(Q_i)$  be a restriction coherent sequence of
cellular algebras over a ring $R$ with field of fractions $F$, and suppose that $Q_i^F$ is split semisimple for all $i$.    Let $\Gamma_i$ denote
the partially ordered set in the cell datum for $Q_i$, and let $\omega_i(\mu, \la)$   denote the inclusion matrix for $Q_i^F \subseteq Q_{i+1}^F$.
	\begin{enumerate}
		[label = {\normalfont (\arabic*)}]
		\item  For $\mu \in \Gamma_{i+1}$ and $\la \in \Gamma_i$, the multiplicity of $\cell
			      \la R $ as a subquotient in any cell filtration of $\Res {Q_{i+1}} {Q_i} {\cell
				      \mu R}$ is $\omega_i(\mu, \la)$.
		\item  Let $K$ be a field that is an $R$-module. Suppose $\cell \mu K $ is a simple
		      cell module of $Q_{i+1}^K$ with semisimple restriction to $Q_i^K$. Then $\Res
			      {Q_{i+1}^K} {Q_{i}^K} {\cell \mu K } \cong \sum^\oplus \omega_i(\mu, \la) \cell
			      \la K $, and the cell modules appearing in this direct sum are simple.
	\end{enumerate}
\end{lemma}

\begin{proof}
	Consider any cell filtration of $\Res {Q_{i+1}}  {Q_i} {\cell \mu  R }$,
	\begin{equation}
		\label{eqn: cell filtration}
		(0) = M_0 \subseteq M_1 \subseteq \cdots \subseteq M_r = \Res {Q_{i+1}}  {Q_i} {\cell   \mu R},
	\end{equation}
	where $M_{k+1}/M_k \cong \cell {\la_k} R $.  Because each $M_k$ is free as an $R$-module,  applying $\otimes_R F$ to the cell filtration
(\ref{eqn: cell filtration}) yields  a filtration of
	$\Res {Q_{i+1}^F}  {Q_i^F} {\cell \mu F }$  with subquotients  $ \cell {\la_k}  F $.
	Hence,  the number of occurrences of $\cell \la R $ in the cell filtration  (\ref{eqn: cell filtration})   is $\omega_i(\mu, \la)$.
	Similarly, applying $\otimes_R K$ to the cell filtration gives a cell filtration of $\Res {Q_{i+1}^K}  {Q_i^K} {\cell  \mu K }$.
	The cell modules that occur as subquotients are semisimple, by assumption, but a semisimple cell module is simple.
\end{proof}

\subsubsection{Order ideals and ideals}  \label{subsection: order ideals} An order ideal in a partially ordered set $(\Gamma, \ge)$ is a subset $U$
such
that if $\mu \in U$ and $\gamma \ge \mu$ in $\Gamma$, then $\gamma \in U$.
Suppose $A$ is a cellular algebra with involution $*$, partially ordered set
$(\Gamma, \ge)$ and cellular basis $\{a_{\mfs, \mft}^\gamma\}$. If $U$ is an
order ideal in $\Gamma$, then the span $A(U)$ of the basis elements $a_{\mfs,
			\mft}^\gamma$ with $\gamma \in U$ is a $*$-invariant ideal in $A$.

\subsection{Traces}\label{sec:traces}   Let $A$ be an algebra over a ring $R$. A {\em trace} on $A$ is an $R$-linear
functional $\tau: A \to R$ that satisfies $\tau(x y) = \tau(yx)$ for all $x, y
	\in A$. A trace $\tau$ is {\em unital} or {\em normalized} if $\tau(\bm 1) =
	1$. The {\em radical} of a trace $\tau$ is the set of $x \in A$ such that
$\tau(x y) = 0$ for all $y \in A$. Evidently, the radical is an ideal. A trace
$\tau$ is {\em faithful} or {\em non-degenerate} if its radical is zero, and
	{\em degenerate} otherwise. If $A\cong \oplus_\la A_\la$ is a split semisimple
algebra over a field, a trace $\tau$ is completely determined by its weight
vector $(\omega_\la)=(\tau(p_\la))$, where $p_\la$ is a minimal idempotent in
$A_\la$, and $\tau$ is non-degenerate precisely if all the weights are
non-zero.

\section{Algebras with a Jones basic construction}  \label{Section abstract semisimplicity theorem} \label{section: axioms}
\subsection{Axioms} \label{axioms}
We give a list of axioms for a tower of algebras $(A_n)_{n \ge 0}$ and a second
tower of algebras $(Q_n)_{n \ge 0}$ that allow application of the method
described in the introduction. In later sections, we show that the families of
algebras mentioned in the introduction satisfy these axioms. In fact, the list
of axioms is simply an abstraction of well--known properties of the Brauer and
BMW algebras. The $q$-Brauer algebras are somewhat different in that the axioms
do not apply to the natural embeddings of the algebras but rather to twisted
embeddings, as will be explained in section \ref{subsection q Brauer}.

We suppose an infinite multiplicity-free branching diagram $\Gamma$ with each
vertex set $\Gamma_k$ endowed with a partial order $\ge$. We write $|\gamma| =
	k$ if $\gamma \in \Gamma_k$. We write $\tilde \Gamma$ for the branching diagram
obtained by reflections from $\Gamma$. We endow $\tilde \Gamma_n$ with the
partial order $(\gamma, n) \ge (\gamma', n)$ if $|\gamma| < |\gamma'|$ or if
$|\gamma| = |\gamma'| = k$ and $\gamma \ge \gamma' $ in $\Gamma_k$. Let $\tilde
	\Gamma_n^0$ be the set of $(\gamma, n) \in \tilde \Gamma_n$ with $|\gamma| <
	n$. Note that $\tilde \Gamma_n^0$ is an order ideal in $\tilde \Gamma_n$.

We consider two increasing sequences of unital algebras $(A_n)_{n \ge 0}$ and
$(Q_n)_{n \ge 0}$ over an integral domain $R$ satisfying the following
conditions:

\begin{enumerate}
	[left=0pt]
	\item  \label{axiom: involution on An}  The algebras $A_n$ and $Q_n$ have algebra involutions denoted by $*$.
	The involutions on the algebras $Q_n$ are consistent; i.e.
	      there is an algebra involution $*$  on $\cup_n Q_n$ such that $(Q_n)^* = Q_n$.
	\item  \label{axiom: A0 and A1}
	      $A_0 = Q_0 = R$   and $A_1 = Q_1$.
	\item \label{axiom: restriction coherent quotients}  The sequence of algebras $(Q_n)_{n \ge 0}$ is a restriction-coherent  tower of cellular
algebras (see section \ref{Coherent towers}).   The partially ordered set in the cell datum of $Q_n$ is $(\Gamma_n, \ge)$.

	\item  \label{axiom: generic ss quotients}    The ground ring $R$ has characteristic zero.  Let $F$ denote its field of fractions.   The
algebras $Q_n^{F} = Q_n \otimes_{R} F$  are split semisimple and the branching diagram for $(Q_n^F)_{n \ge 0}$  is $\Gamma$.
	Moreover, for any field $K$ that is an $R$-module, either $Q_n^K = Q_n \otimes_R K$ is split semisimple for all $n$  or  there exists an $n_1
\ge 0$ such that $Q_n^K$
	      is split semisimple  for $n \le n_1$ and non-semisimple for $n >n_1$.
	\item  \label{axiom: cellular Axn}  $A_n$ is cellular with partially ordered set $(\tilde \Gamma_n, \ge)$.
	\item \label{axiom:  idempotent and Qn as quotient of An}
	      For $n \ge 1$,  $A_{n+1}$ contains an  essential idempotent $e_{n}$ such that $e_{n}^2 = \bm \delta e_n$,  with $\bm \delta$ invertible
in $R$.   The ideal
	      $A_{n+1} e_{n} A_{n+1}$ in $A_{n+1}$  coincides with the ideal $A_{n+1}(\tilde \Gamma_n^0)$  corresponding to the order ideal $\tilde
\Gamma_n^0$ in $\tilde \Gamma_n$.  (See Section \ref{subsection: order ideals}.)   Moreover,
	      $A_{n+1}/(A_{n+1} e_{n} A_{n+1})  \cong Q_{n+1}$.
	\item  \label{axiom: TL relations} The elements $e_n$ satisfy the Temperley-Lieb relations $e_n e_{n\pm 1} e_n = e_n$  and $e_n e_m = e_m
e_n$ if $|n-m| \ge 2$.
	\item \label{axiom: en An en} For $n \ge 1$,   $e_{n}$ commutes with $A_{n-1}$ and $e_{n} A_{n} e_{n} \subseteq  A_{n-1} e_{n}$.
	\item  \label{axiom:  An en}
	      For $n \ge 1$,  $A_{n+1} 	e_{n} = A_{n} e_{n}$,  and the map $x \mapsto x e_{n}$ is injective from
	      $A_{n}$ to $A_{n} e_{n}$.
\end{enumerate}

\smallskip
Define $\bar e_n = (1/\bm \delta) e_n$.   Then $\bar e_n$ is an idempotent.  For $s \le n/2$ define
$$\ex {n, s} = e_{n - 2s +1} e_{n - 2s +3} \cdots e_{n-1} \quad (s \ \ \text{factors})
$$
and $\barex {n, s} = (1/\bm \delta^s) \ex {n, s}$, which is an idempotent.

It follows from axioms (\ref{axiom: en An en}) and (\ref{axiom: An en}) that
there is a map $E^n_{n-1}$ from $A_n$ to $A_{n-1}$ such that $\bar e_n x \bar
	e_n = \bar e_n E^n_{n-1}(x)$ for $x \in A_n$. Moreover, $E^n_{n-1}(\bm 1) = \bm
	1$ and $E^n_{n-1}(a x b) = a E^n_{n-1}(x) b$ for $x \in A_n$ and $a, b \in
	A_{n-1}$. A map with these properties -- i.e., a unital $A_{n-1}$--$A_{n-1}$
bimodule map -- is often called a {\em conditional expectation}. It follows
from axiom (\ref{axiom: TL relations}) that $E^n_{n-1}( \bar e_{n-1}) = 1/\bm
	\delta^2$. We define a functional $\tr: \cup_n A_n \to R$ recursively by
$\tr(\bm 1) = 1$ and $\tr(x) = \tr(E^n_{n-1}(x))$ for $x \in A_n$. We have
$\tr(\bar e_n) = (1/\bm \delta^2)$, and for $x \in A_n$,
\begin{equation}
	\label{bimoduleproperty}
	\tr(\bar e_n x)  = \tr(E^{n+1}_n(\bar e_n x)) = (1/\bm \delta^2) \tr(x),
\end{equation}
by the bimodule property of $E^{n+1}_n$.

Axioms (\ref{axiom: en An en}) and (\ref{axiom: An en}) also hold in
specializations $A_n^S = A_n \otimes_R S$; in particular injectivity of $x
	\mapsto x e_n$ from $A_n^S$ to $A_n^S e_n$ follows from freeness over $R$ of
$A_n$ and $A_{n+1}$. Consequently there also exist conditional expectations
$E^n_{n-1}: A_n^S \to A_{n-1}^S$ and a functional $\tr: \cup_n A_n^S \to S$
defined by the same recursion as above; moreover, $ E^n_{n-1}(x \otimes_R S) =
	E^n_{n-1}(x) \otimes_R S$ and $\tr(x \otimes_R S) = \tr(x) \otimes_R S$.

\smallskip
\begin{enumerate}
	[left=0pt, resume]
	\item  \label{axiom: trace}  The functional $\tr$ is a  trace, i.e. $\tr( xy ) = \tr(yx)$ for $x, y \in A_n$.
	\item \label{axiom: faithfulness of tr}    The trace $\tr$ on $A_n^F$  is faithful  for all $n$.
\end{enumerate}

The last axiom connects the cellular structure of $A_n$ with that of
$A_{n-2s}$, so it is another coherence axiom.

\begin{enumerate}
	[left=0pt, resume]
	\item  \label{axiom:  weak coherence of An}  Suppose $k < n$ with $n - k$ even and let $s = (n-k)/2$.   Let $(\gamma, k) \in \tilde
\Gamma_k$.
	      If $x \in A^{\ge (\gamma, k)}_k \setminus A^{> (\gamma, k)}_k$, then
	      $x \ex {n, s} = x e_{k+1} e_{k+3} \cdots e_{n-1} \in A^{\ge (\gamma, n)}_n \setminus A^{> (\gamma, n)}_n$.
\end{enumerate}

\smallskip
This completes our list of axioms.

\begin{lemma}
	\label{lemma: markov trace}  Axiom {\normalfont(\ref{axiom: trace})}  is equivalent to the following statement:

	There is a unique trace $\tr: \cup_n A_n \to R$ such that $\tr(\bm 1) = 1$ and
	for all $n$ and for $x \in A_n$, $\tr(\bar e_n x) = (1/\bm \delta^2) \tr(x)$.
\end{lemma}

\begin{proof}
	Suppose that $\tau$ is a trace with the stated properties.  Let $x \in A_n$.  Then:
	$$
		(1/\bm \delta^2)  \tau(x) =  \tau(\bar e_n x) = \tau(\bar e_n x \bar e_n) = \tau(\bar e_n E^n_{n-1}(x)) = (1/\bm \delta^2) \tau(
E^n_{n-1}(x))
	$$
	Therefore $\tau(x) = \tau( E^n_{n-1}(x))$.
\end{proof}

\begin{remark}
	\label{remark: e n s A n e n s} It follows from axioms (\ref{axiom: en An en})   and (\ref{axiom:  An en})  that   for $s \le n/2$
	$$
		\ex {n, s}  A_n \ex {n,s} \subseteq  \ex {n, s}  A_{n-2s},
	$$
	and $\ex {n, s} $  commutes with $A_{n-2s}$.

\end{remark}

\subsection{Semisimplicity theorem -- sufficient condition}
In this section, we prove a sufficient condition for semisimplicity for
algebras that fit our framework, following the method of ~\cite{Wenzl-Brauer}
and~\cite{Wenzl-BCD}.

In the following items \ref{A k semisimple cellular} through \ref{quotient by
	rad of trace}, we assume that $(A_n)_{n \ge 0}$ and $(Q_n)_{n \ge 0}$ are two
sequences of algebras over an integral domain $R$ satisfying axioms
	{\normalfont (\ref{axiom: involution on An})} through {\normalfont (\ref{axiom:
			weak coherence of An})}, and that $K$ is a field that is an $R$-module.
\begin{remark}
	\label{A k semisimple cellular}  If $A_{n}^K$ is split semisimple for some $n$,  then the simple modules of  $A_{n}^K$  are the cell modules
$\Delta^K_{(\gamma, n)}$,  where ${(\gamma, n)} \in \tilde \Gamma_n$.    We let $z_{(\gamma, n)}$ denote the minimal central idempotent such that
	$$z_{(\gamma, n)} A_{n}^K   \cong \End_K(\Delta^K_{(\gamma, n)}).$$
\end{remark}

\begin{lemma}
	\label{min idempotent lemma}   Let $n = k + 2s$ and assume $A_k^K$  and $A_n^K$   are both split semisimple.   If $p_0$ is a minimal idempotent
in
	$z_{(\gamma, k)}  A_k^K$,  then $p = \barex {n, s} p_0$ is a minimal idempotent in $z_{(\gamma, n)}  A_n^K$.  Moreover,
	$$\tr(p) = (1/\delta^{2s}) \tr(p_0).$$
\end{lemma}

\begin{proof}   Since $K$ is fixed, we will write $A_k$  for $A_k^K$, etc.
	Evidently, $p$ is an idempotent.   First, we check that $p$ is a minimal idempotent in $A_n$.    Since $\barex {n, s} A_n \barex {n, s}
\subseteq  \barex {n,s}  A_k$,  by Remark \ref{remark: e n s A n e n s},
	and since $p_0$ is a minimal idempotent in $A_k$,  it follows that $p A_n p = K p$.
	By Lemma \ref{semisimple cellular algebras} part (4), $p_0 \in A_k^{\ge (\gamma, k)} \setminus A_k^{> (\gamma, k)} $, so by  axiom (\ref{axiom:
weak coherence of An}),
	$p \in
		A_n^{\ge (\gamma, n)} \setminus A_n^{> (\gamma, n)} $.
	Therefore, $p z_{(\gamma, n)} \ne 0$, again by Lemma  \ref{semisimple cellular algebras}.  But since $p$ is a minimal idempotent, $p = p
z_{(\gamma, n)} $.
	The formula  $\tr(p) = (1/\delta^{2s}) \tr(p_0)$ follows from the remarks following the statement of axiom (\ref{axiom:  An en}).
\end{proof}

\begin{lemma}
	\label{quotient by rad of trace}   Let $K$ be a field that is an $R$-module.
	Let $I^K_n$ denote the radical of the trace $\tr$ on $A_n^K$,  let $\bar A^K_n =  A^K_n/I^K_n$, and let
	$\pi: A^K_n \to \bar A^K_n$ denote the quotient map.
	\begin{enumerate}
		[label = {\normalfont (\arabic*)}]
		\item  $I^K_{n-1} \subseteq I^K_n$.  Hence, $\bar A^K_{n-1}$ embeds in $\bar A^K_n$.
		\item  $E^n_{n-1}(I^K_n) = I^K_{n-1}$.
		      Hence, there is a conditional expectation $E^n_{n-1} : \bar A^K_n \to \bar A^K_{n-1}$  such that $E^n_{n-1} (\pi(x)) =
\pi(E^n_{n-1} (x))$ for $x \in A^K_n$.
		\item  There is a trace $\tr$ on $\cup \bar A^K_n \to K$ such that $\tr(\pi(x)) =
			      \tr(x)$ for $x \in A_n$.
	\end{enumerate}
\end{lemma}

\begin{proof}
	All the statements are straighforward to check.  For example,  if $x \in I_{n-1}^K$ and $y \in A_n$, then
	$$
		\tr(x y) = \tr(E^{n}_{n-1}(xy)) = \tr(x E^{n}_{n-1}(y)) = 0.
	$$
\end{proof}

Suppose $A \subseteq B \subseteq C$ are unital associative algebras over a
field with the following properties:
\begin{enumerate}
	\item There exists a trace $\tr$ on $B$ and a conditional expectation $E:B\to A$ with
	      $\tr(b a ) = \tr(E(b) a)$ for all $b\in B$ and $a \in A$.
	\item There exists an idempotent $\eps \in C$ with $\eps b \eps = E(b) \eps$ for $b
		      \in B$ and $\eps a = a \eps$ for $a \in A$.
	\item $a \mapsto a \eps$ is injective from $A $ to $C$.
	\item $C \eps C= B \eps B$.
\end{enumerate}

The following is a crucial lemma from ~\cite{Wenzl-Brauer} that allows
recognizing a ``Jones basic construction" inside of $C$ in the situation just
described, i.e. $C \eps C \cong \End(B_A)$.

\begin{lemma}
	\label{Wenzl 3 algebra lemma}  Let $A \subseteq B \subseteq C$, $\tr$, $E$, $\eps$ be as above.  Suppose in addition that $A$ and $B$ are split
semisimple, and that
	$\tr$ is faithful on $A$ and on $B$.
	\begin{enumerate}
		[label = {\normalfont (\arabic*)}]
		\item $C \eps C = B \eps B \cong \End(B_A)$,  i.e. the endomorphisms of $B$ as a right $A$-module.
		\item $B \eps B = C \eps C$ is a unital split semisimple algebra with minimal ideals in one-to-one correspondence with those of $A$.
The correspondence is such that if $p$ is a minimal idempotent of $A$ in some minimal ideal, then $p \eps$ is a minimal idempotent in $B \eps B$
in the corresponding minimal ideal.
		\item  $B$ embeds in $\End(B_A) \cong   B \eps B$ by $b \mapsto$ left multiplication by $b$.  Hence, there is a  monomorphism $B \to B
\eps B$.   The inclusion matrix for this monomorphism is the transpose of that for the inclusion $A \subseteq B$.
	\end{enumerate}
\end{lemma}

\begin{proof}
	This follows e.g.\  from ~\cite{Wenzl-Brauer}, Proposition 1.2.
\end{proof}

The following theorem gives a sufficient condition for semisimplicity for
algebras that fit our framework. The proof follows the approach in
~\cite{Wenzl-Brauer} and ~\cite{Wenzl-BCD} where the statement was proved for
the special cases of Brauer and BMW algebras.

\begin{theorem}
	\label{first semisimplicity theorem - axiomatic version}
	Let $(A_n)_{n \ge 0}$ and $(Q_n)_{n \ge 0}$ be sequences of algebras satisfying the axioms {\normalfont (\ref{axiom: involution on An})}
through {\normalfont (\ref{axiom:  weak coherence of An}).}  Let $K$ be a field that is an $R$-module.
	Let $n_0$ be the smallest value of $n$ such that  the  trace $\tr$ on $A_n^K$ is degenerate, with the convention that $n_0 = \infty$ if $\tr$
is non-degenerate
	on $A_n^K$ for all $n$.
	Let $n_1 + 1$ be the smallest value of $n$ such that $Q_n^K$ fails to be semisimple, again with the convention that $n_1 = \infty$ if $Q_n^K$
is semisimple for all $n$.

	Then $A_k^K$ is split semisimple for all $k \le m = \min\{n_0, n_1\}$.
	Moreover, the branching diagram for the sequence of algebras $(A_k^K)_{k \le
				m}$ is $\tilde \Gamma$ (truncated at level $m$).
\end{theorem}

\begin{proof}
	Since $K$ is fixed, we will write $A_k$  for $A_k^K$ and    $Q_k$  for $Q_k^K$.   The proof is by induction on $k$.  The strategy is to apply
Lemma \ref{Wenzl 3 algebra lemma} to recognize the Jones basic construction for $A_{k-1}  \subseteq A_k$ inside $A_{k+1}$.

	The assertions are evident for $k=0$ and $k=1$, using axiom (\ref{axiom: A0 and
		A1}). Assume the assertion holds for the algebras up to a certain level $k < m$
	and consider the inclusions $A_{k-1} \subseteq A_k \subseteq A_{k+1}$. By
	assumption, the trace is non-degenerate on $A_{k-1}$ and on $A_k$, and by the
	induction hypothesis, $A_{k-1}$ and $A_k$ are split semisimple. The remaining
	hypotheses of Lemma \ref{Wenzl 3 algebra lemma} hold, with $\eps = \bar e_{k}$
	and $E = E^{k}_{k-1}$, using axioms (\ref{axiom: en An en}), (\ref{axiom: An
		en}), and (\ref{axiom: trace}) and the discussion following axiom (\ref{axiom:
		An en}). Therefore, the ideal $A_{k+1} e_{k} A_{k+1} = A_{k} e_{k} A_{k}$ in
	$A_{k+1}$ is unital split semisimple. By axiom (\ref{axiom: idempotent and Qn
		as quotient of An}), the quotient $A_{k+1}/A_{k+1} e_{k} A_{k+1}$ is isomorphic
	to $Q_{k+1}$, which is split semisimple by hypothesis. Hence $A_{k+1}$ is split
	semisimple.

	It remains to check the assertion about the branching diagram for $A_k
		\subseteq A_{k+1}$. For $j \le k+1$, let $z_j$ be the central idempotent in
	$A_j$ such that $z_j A_j = A_j e_{j-1} A_j$. By the induction hypothesis, the
	branching diagram for $A_{k-1} \subseteq A_k$ is the $k$-th level of $\tilde
		\Gamma$. According to Lemma \ref{Wenzl 3 algebra lemma}, the simple components
	of $z_{k+1} A_{k+1}$ correspond one-to-one with those of $A_{k-1}$ and the
	branching diagram for $A_k \to z_{k+1}A_{k+1}$ is the reflection of that for
	$A_{k-1} \subseteq A_k$. Hence, we can label the simple components of $z_{k+1}
		A_{k+1}$ by pairs $(\gamma, k+1)$, where $(\gamma, k-1) \in
		\tilde\Gamma_{k-1}$, and the branching for $A_k \to z_{k+1} A_{k+1}$ is
	specified by $(\mu, k) \nearrow (\gamma, k+1)$ if and only if $(\gamma, k-1)
		\nearrow (\mu, k)$. Finally, we have to determine the branching for $A_k \to
		(1-z_{k+1}) A_{k+1}$. Because of the Temperley-Lieb relations (axiom
	(\ref{axiom: TL relations})), $\bar e_1, \dots, \bar e_{k}$ are mutually
	equivalent idempotents in $A_{k+1}$, and therefore generate the same ideal in
	$A_{k+1}$. Thus $A_k \bar e_{k-1} A_k \subseteq A_{k+1} \bar e_{k-1} A_{k+1} =
		A_{k+1} \bar e_{k} A_{k+1} $; hence, $z_k (1-z_{k+1}) = 0$. Moreover, $(1-z_j)
		A_j \cong Q_j$ for $j \le k+1$, and the branching diagram for $(1-z_k) A_k \to
		(1- z_{k+1}) A_{k+1}$ is the same as that for $Q_k \subseteq Q_{k+1}$. By Lemma
	\ref{lemma: permanent multiplicities} and axioms (\ref{axiom: restriction
		coherent quotients}) and (\ref{axiom: generic ss quotients}), it is given by
	the $(k+1)$-st level of $\Gamma$. Combining these observations, the branching
	diagram for $A_k \subseteq A_{k+1}$ coincides with the $(k+1)$-st level of
	$\tilde \Gamma$.
\end{proof}

\begin{corollary}
	\label{corollary: generic semisimplicity}   \phantom{xxx}
	\begin{enumerate}
		[label = {\normalfont (\arabic*)}]
		\item
		      Let $(A_n)_{n \ge 0}$ and $(Q_n)_{n \ge 0}$ be sequences of algebras satisfying the axioms {\normalfont (\ref{axiom: involution on
An})}  through {\normalfont (\ref{axiom:  weak coherence of An}).}   Let $F$ denote the field of fractions of the generic ground ring $R$ for the
algebras $A_n$.    Then $A_n^F$ is split semisimple for all $n$, and the branching diagram for $(A_n^F)_{n\ge 0}$ is $\tilde \Gamma$.
		\item    Let $\omega_{(\gamma, n)}$ denote the trace of a minimal idempotent in
		      $z_{(\gamma, n)} A_n^F$. If $|\gamma| = k$ and $n = k + 2s$, then we have
		      $\omega_{(\gamma, n)} = \omega_{(\gamma, k)}/\bm \delta^{2s}.$ Hence, writing
		      $d^F_\gamma = \bm \delta^k \omega_{(\gamma, k)}$, we have $\omega_{(\gamma, n)}
			      = d^F_\gamma/ \bm \delta^n$.

	\end{enumerate}
\end{corollary}

\begin{proof}
	Part (1) follows from Theorem \ref{first semisimplicity theorem - axiomatic version}  and axioms  (\ref{axiom: generic ss quotients}) and
(\ref{axiom: faithfulness of tr}).   Part (2) follows from Lemma \ref{min idempotent lemma}.
\end{proof}

\subsection{Semisimplicity theorem -- necessary and sufficient condition}
In this section, we use the cellular structure of the algebras $A_n$, in
particular axiom (\ref{axiom: weak coherence of An}), to show that the
sufficient condition for semisimplicity of Theorem \ref{first semisimplicity
	theorem - axiomatic version} is also necessary.

Adopt the hypotheses of Theorem \ref{first semisimplicity theorem - axiomatic
	version}. For the rest of this section, we will be considering algebras
$A_n^K$, $Q_n^K$, etc. over a fixed field $K$, so we will drop the superscript
$K$ from our notation, except in the formal statement of results. Also, we will
assume $n_0<n_1$, and $n\leq n_1$.

Let $I_n$ denote the radical of the trace on $A_n$ and $\bar A_n = A_n/I_n$.
Let $\pi : A_n \to \bar A_n$ denote the quotient map. Let $J_n = A_n e_{n-1}
	A_n$ and let $\rho : A_n \to Q_n = A_n/J_n$ denote the quotient map. Let
$\overline Q_n$ denote $$ \overline Q_n = A_n/(I_n + J_n) \cong \bar
	A_n/\pi(J_n) \cong Q_n/\rho(I_n). $$ As $Q_n$ is semisimple with simple
components $Q_{n,\gamma}$ labeled by $\gamma \in \Gamma_n$, we obtain a subset
$\Gamma(K)_n\subset \Gamma_n$ consisting of all labels $\gamma$ such that
$Q_{n,\gamma}\subset \overline Q_n$. We define $\Gamma(K)$ to be the subgraph
of $\Gamma$ whose vertices are the ones in $\cup_{n\leq n_1} \Gamma(K)_n$, with
the edges inherited from $\Gamma$. We then have

\begin{proposition}
	\label{semisimplicity of quotients by the radical} Assume $n_0< n_1$ and $n\leq n_1$. Then $\bar A_n^K$ is split semisimple, and the branching
diagram for the sequence of algebras $(\bar A_n^K)_{n\leq n_1}$ is $\tilde \Gamma(K)$. In particular, $\Gamma(K)_{n_0}\neq \Gamma_{n_0}$.
\end{proposition}

\begin{proof}
	It follows from Lemma \ref{quotient by rad of trace} that we have well-defined unital embeddings $\bar A_n\subset \bar A_{n+1}.$
	We can
	use the proof of Theorem \ref{first semisimplicity theorem - axiomatic version} to prove our claim  for this sequence of algebras for $n\leq
n_1$, by induction on $n\leq n_1$.
	It is obviously true for $n\leq n_0$. Assuming the claim for $\bar A_{n-1}\subset \bar A_n$, it follows that $\bar J_{n+1}=\bar A_n\bar e_n\bar
A_n$ is semisimple and its simple summands are labeled by the ones of $\bar A_{n-1}$,
	using Lemma \ref{Wenzl 3 algebra lemma}. In particular, the branching $\bar A_n\subset \bar J_{n+1}$ is the reflection of the branching of
	$\bar A_{n-1}\subset \bar A_n$. Moreover, $\bar A_{n+1}/\bar J_{n+1}\cong \overline Q_{n+1}$. This proves that the simple components of
	$\bar A_{n+1}$ are labeled by $\bigcup_{0\leq k\leq \lfloor (n+1)/2\rfloor} \Gamma(K)_{n+1-2k}$,
	as claimed. As a simple $\overline Q_{n+1}$ module is annihilated by $J_n\subset J_{n+1}$, its restriction to $\bar A_n$ must be a direct
	sum of simple $Q_n$ modules.
	Indeed, each of them must also be a simple $\bar Q_n$ module, as the embedding $\bar A_n\subset \bar A_{n+1}$  is unital.

	To prove the last claim, it follows from our arguments before that
	$A_{n_0}\cong J_{n_0}\oplus Q_{n_0}$ is semisimple. As $\tr$ is non-degenerate
	on $A_{n_0-2}$, it is so also on $J_{n_0}$, by Lemma \ref{min idempotent
		lemma}, with $n=n_0$ and $s=1$. Hence there must be a simple component in
	$Q_{n_0}$ that is annihilated by $\tr$.
\end{proof}

\begin{corollary}
	\label{bookkeeping lemma -abstract setting}   Adopt the hypotheses  and notation of Proposition \ref{semisimplicity of quotients by the
radical}.
	Let $\pi : A_n^K \to \bar A_n^K$  denote the quotient map.
	Suppose
	$$k < n_0 <n = k+2s \le n_1,$$
	so $I_k^K = (0)$  and $A_k^K = \bar A_k^K$ is semisimple.
	Suppose $\gamma \in \Gamma_k$ and $p_0$ is a minimal idempotent in $z_{(\gamma, k)}A_k^K$ .   Then
	$\pi(\barex {n, s} p_0)$ is a minimal idempotent in the simple component of $\bar A_n^K$   labelled by $(\gamma, n) \in \tilde \Gamma(K)$.
\end{corollary}

\begin{proof}
	This follows from the proof of Proposition \ref{semisimplicity of quotients by the radical} and by  repeated application of  Lemma \ref{Wenzl 3
algebra lemma}, part (2).
\end{proof}

The following is our main theoretical result, providing a necessary and
sufficient condition for semisimplicity of algebras that fit our framework.

\begin{theorem}
	\label{second semisimplicity theorem - axiomatic version}
	Adopt the hypotheses of Theorem \ref{first semisimplicity theorem - axiomatic version}.  Then $A_n^K$ is  semisimple if and only if
	$n \leq m = \min\{n_0, n_1\}$.
\end{theorem}

\begin{proof}
	We drop the superscript $K$ from $A_n^K$,  $Q_n^K$, etc. It was already shown in Theorem \ref{first semisimplicity theorem - axiomatic version}
that the algebras $A_n$ are semisimple if $n\leq m$.

	If $n > n_1$, then $A_n$ is non-semisimple because the quotient algebra $Q_n$
	is non-semisimple, by axiom (\ref{axiom: generic ss quotients}). So we can
	suppose $n_0 < n \le n_1$. In order to reach a contradiction, suppose that
	$A_n$ is semisimple. We will show that the dimension of a certain simple
	component of $A_n$ can be computed in two different ways with different results
	-- a contradiction that shows that $A_n$ cannot be semisimple.

	By Theorem \ref{first semisimplicity theorem - axiomatic version}, $A_{n_0}$ is
	semisimple; but there exists $\gamma_0\in \Gamma_{n_0}\setminus
		\Gamma(K)_{n_0}$ by Proposition \ref{semisimplicity of quotients by the
		radical}. There exists a $(\gamma, n) \in \tilde \Gamma_n$ that is connected by a path in
	the branching diagram $\tilde \Gamma$ to $(\gamma_0, n_0)$, and such that $k =
		|\gamma| < n_0$. Write $s = (n-k)/2$. $A_k$ is split semisimple according to
	Theorem \ref{first semisimplicity theorem - axiomatic version}, and the weights
	of the trace on $A_k$ are non-zero by assumption. Let $p_0$ be a minimal
	idempotent in $z_{(\gamma, k)} A_k$. Then by Lemma \ref{min idempotent lemma},
	$p = \barex {n, s} p_0$ is a minimal idempotent in $z_{(\gamma, n)} A_n$, and
	$\tr(p) = (1/\delta)^{2s} \tr(p_0) \ne 0$. In particular, $p \not\in I_n$.
	Since $z_{(\gamma, n)} A_n$ is simple, $I_n \cap z_{(\gamma, n)} A_n$ is either
	zero or equal to $z_{(\gamma, n)} A_n$; we conclude that it is zero.

	Therefore, $\pi : A_n \to \bar A_n = A_n/I_n$ maps $z_{(\gamma, n)} A_n$
	isomorphically onto some simple component of $\bar A_n $; by Lemmas \ref{min
		idempotent lemma} and \ref {bookkeeping lemma -abstract setting} the image is
	(unsurprisingly) the simple component labelled by $(\gamma, n)$. Because
	$z_{(\gamma, n)} A_n \cong \End(\Delta_{(\gamma, n)})$, where $\Delta_{(\gamma,
			n)}$ denotes the cell module, $\dim(z_{(\gamma, n)} A_n) =
		\dim(\Delta_{(\gamma, n)})^2$. The rank of the cell module is independent of
	the ground ring, and over the field of fractions $F$ of the generic ground ring
	$R$, it is the number $d$ of paths from the root to $(\gamma , n)$ in the
	branching diagram $\tilde \Gamma$, using Corollary \ref{corollary: generic
		semisimplicity} . On the other hand, the dimension of the corresponding simple
	component of $\bar A_n$ is strictly less. It is $d'^2$, where $d'$ is the
	number of paths on the restricted branching diagram $\tilde \Gamma(K)$ of
	Proposition \ref{semisimplicity of quotients by the radical} from the root to
	$(\gamma, n)$. This is strictly less than $d$, because any path through
	$(\gamma_0, n_0)$ is missing. This contradiction shows that $A_n$ is not
	semisimple.
\end{proof}

\subsection{Determining  \boldmath $ {m = \min\{n_0, n_1\} }$\unboldmath} \label{sec:trace} It remains to find an efficient way to determine $m =
\min\{n_0, n_1\}$.  In our examples, explicit formulas are known for  the trace evaluated at minimal idempotents for the algebras defined over
$F$,  the field of fractions of the generic ground ring.  We will show that these formulas can still be used for specializations over a field $K$
as long as our algebras remain semisimple over $K$.   That is, we will show that as long as the algebras $A_n^K$ remain semisimple, then
the weights of the trace on $A_n^F$  are ``evaluable"  in $K$, and the evaluations give the weights of the trace in $A_n^K$.
This allows us to determine $m$ by looking at vanishing conditions for the weights over $K$.

Here is our general set-up in more detail: Let $A$ be an algebra over an
integral domain $R$, with $A$ free as an $R$-module. Let $F$ denote the field
of fractions of $R$. Suppose $K$ is a field that is an $R$-module; that is
there exists a ring homomorphism $\theta : R \to K$. Let $\mathfrak I$ denote
the kernel of $\theta$, a prime ideal in $R$. Then we have the auxiliary rings
$R_\mathfrak I \subseteq F$, the localization of $R$ at $\mathfrak I$, and
$R/\mathfrak I$, the image of $\theta$ in $K$. The homomorphism $\theta$
extends to $\theta: R_\mathfrak I \to K$. We call elements of $A \otimes_R
	R_\mathfrak I \subseteq A \otimes_R F$ the {\em $K$-evaluable elements} of $A
	\otimes_R F$. Abusing language, we will sometimes refer to $K$-evaluable
elements of $A$, meaning $K$-evaluable elements of $A^F$. If $x$ is
$K$-evaluable, then $x \otimes_{R_\mathfrak I} K$ is a well defined element of
$A^K$. If $\{a_i\}$ is any $R$-basis of $A$, then $A \otimes_R R_\mathfrak I$
is the set of $R_{\mathfrak I}$ linear combinations of the basis elements,
independent of the choice of the basis. If $x = \sum_i \alpha_i a_i$ is any
such element, then $x \otimes_{R_\mathfrak I} K = \sum_k \theta(\alpha_i) a_i$.

We review well-known facts about calculating idempotents and matrix units for
certain classes of algebras.

\begin{definition}
	\label{matrixunitdef} \phantom{xxx}
	\begin{enumerate}
		[label = {\normalfont (\alph*)}]
		\item Let $A\subset B$ be finite-dimensional split-semisimple algebras over a field.
		      We say that a collection of minimal idempotents $(p_t)$ is a partition of unity
		      if $\sum_t p_t=1$ and $p_tp_{\tilde t}=p_{\tilde t}p_t=0$ for $t\neq \tilde t$.
		      We say that a partition of unity $(q_s)\subset B$ is a refinement of $(p_t)$ if
		      each $q_s$ is a subidempotent of some $p_t$, i.e. $p_tq_s=q_s=q_sp_t$.
		\item We say that a system of matrix units $(f_{t,s}^\mu)$ of $B$ is a refinement of
		      a system of matrix units $(e_{i,j}^\lambda)$ of $A$ if $(f_{t,t}^\mu)$ is a
		      refinement of $(e_{i,i}^\lambda)$. In this case, we call $(e_{i,j}^\lambda)$
		      and $(f_{t,s}^\mu)$ compatible systems of matrix units.
	\end{enumerate}
\end{definition}

\begin{lemma}
	\label{inductionargument}   \label{evaluable matrix unit lemma}
	Let $R$, $F$, $K$, etc. be as in the discussion above. Suppose that  $A \subseteq B$ are $R$-algebras that are free as $R$-modules, that
	$A^F$ and $B^F$  are split semisimple,  that the inclusion $A^F \subseteq B^F$ is multiplicity free (see Section \ref{section: branching
	diagrams}),  that $A^K$ is split semisimple, and that $A^F$ has  a system of $K$-evaluable  matrix units $(e_{i, j}^\lambda)$.

	Assume that for each minimal central idempotent $z_\mu$ in $B^F$ and for each
	minimal central idempotent $z'_\la$ in $A^F$ such that $z'_\la z_\mu \ne 0$,
	there exists a $K$-evaluable minimal idempotent $p_\la$ in $B^F$ such that
	$$p_\lambda = e^\la_{i_0, i_0} p_\la = e^\la_{i_0, i_0} z_\mu$$ for some $i_0.$
	Moreover, assume that for each pair $z'_\lambda$, $z'_\nu$ such that $z'_\la
		z_\mu \ne 0$ and $z'_\nu z_\mu \ne 0$, there exist $K$-evaluable elements $h,
		h'$ in $B^F$ such that $ p_\la h p_\nu h' p_\la \otimes_{R_\mathfrak I} K \ne 0
	$.

	Then $B^F$ has a $K$-evaluable system of matrix units $(f_{s, t}^\mu)$ refining
	$(e_{i, j}^\lambda)$.
\end{lemma}

\begin{proof}
	It suffices to show that $z_\mu B^F$   has a $K$-evaluable system of matrix units refining  $(e_{i, j}^\lambda)$, for a fixed minimal central
idempotent $z_\mu$ of $B^F$.
	Let $L$  denote the set  of minimal central idempotents  $z'_\lambda$  of $A^F$ such that
	$z'_\lambda z_\mu \ne 0$.    For $z'_\lambda \in L$, define
	$$f^\mu_{(i, \lambda), (j, \lambda)} =  e^\lambda_{i, j} z_\mu =  e^\la_{i, i_0}  e^\la_{i_0, i_0} z_\mu  e^\la_{i_0, j} = e^\la_{i, i_0}
p_\la e^\la_{i_0, j}.$$
	This is a complete system of $K$-evaluable matrix units in $z'_\la z_\mu B^F z'_\la$.

	Impose an arbitrary total order on $L$. By hypothesis, for each $z'_\la <
		z'_\nu$ in $L$, there exist $K$-evaluable elements $h, h'$ in $B^F$ such that
	$p_\la h p_\nu h' p_\la \otimes_{R_\mathfrak I} K \ne 0$. Since $p_\la$ is a
	minimal idempotent, necessarily $ p_\la h p_\nu h' p_\la = \alpha p_\la$, where
	$\alpha$ is $K$-evaluable and $\alpha^K$ is non-zero. It follows that $p_\nu h'
		p_\la h p_\nu= \alpha p_\nu$ as well.

	If $j_0$ is such that $p_\nu = e^\nu_{j_0, j_0} z_\mu$, define $$f^\mu_{(i_0,
			\la), (j_0, \nu)} = p_\la h p_\nu \quad \text{ and} \quad f^\mu_{(j_0, \nu),
			(i_0, \la)} = (1/\alpha) p_\nu h' p_\la.$$ Then $$ f^\mu_{(i_0, \la), (i_0,
			\la)}, f^\mu_{(i_0, \la), (j_0, \nu)}, f^\mu_{(j_0, \nu), (i_0, \la)},
		f^\mu_{(j_0, \nu), (j_0, \nu)} $$ is a $2$-by-$2$ system of $K$-evaluable
	matrix units. We complete the system of $K$-evaluable matrix units by defining
	for $\la \ne \nu$ in $L$, $$ f^\mu_{(i, \la), (j, \nu)} = f^\mu_{(i, \la),
			(i_0, \la)} f^\mu_{(i_0, \la), (j_0, \nu)} f^\mu_{(j_0, \nu), (j, \nu)}. $$
\end{proof}

\begin{lemma}
	\label{Anmatrixunits} Let $(A_n)_{n \ge 0}$ and $(Q_n)_{n \ge 0}$ be sequences of algebras satisfying the
	axioms {\normalfont (\ref{axiom: involution on An})}  through {\normalfont (\ref{axiom:  weak coherence of An}).}  Let $K$ be a field that is
an
	$R$-module.  Assume the algebras $Q_n$ have compatible systems of matrix units that are $K$-evaluable for $n\leq n_1$. Then we obtain
	compatible systems of  $K$-evaluable matrix units for $A_n$ if  $n\leq m=\min\{n_0, n_1\}$.
\end{lemma}

\begin{proof}
	We follow the proof in  (\cite{Ram-Wenzl}, Theorem 1.4), which   goes by induction on $n$, with the claim obviously true for $n=0,1$.
	Let $n <m$  and suppose the claim holds for $A_k$, for $k \le n$.   In particular, we have compatible systems of $K$-evaluable matrix units in
$A_k^F$  for $k \le n$;   and the weights of the trace  on $A_k^F$ are given by $(\omega_{(\gamma,k)})=(d^F_\gamma/\bm \delta^{k})$,  with
$d^F_\gamma$ $K$-evaluable and $d_\gamma^K$ non-zero.

	We know that $A_{n+1}^F$ is split semisimple. Let $z_{(\mu, n+1)}$ be a minimal
	central idempotent in $A^F_ne_nA^F_n$, and let $z_{(\mu, n-1)}$ be the
	corresponding minimal central idempotent in $A_{n-1}^F$; see Lemma \ref{Wenzl 3
		algebra lemma}. Let $p_t$ and $p_s$ be $K$-evaluable minimal idempotents in
	$A_n^F$ such that there exists a $K$-evaluable minimal idempotent $\tilde
		p_r\in z_{(\mu, n-1)}A^F_{n-1}$ with $p_s\tilde p_r =p_s$, $\tilde p_rp_t=p_t$.
	Suppose moreover that $p_t$, $p_s$, and $\tilde p_r$ are diagonal matrix units
	in the compatible system of matrix units given by the induction hypothesis.
	Suppose that $p_t$ and $p_s$ belong respectively to the minimal ideals in
	$A_n^F$ labelled by $(\lambda(t), n)$ and $(\lambda(s), n)$. Then it follows
	from the bimodule property and the fact that $\tilde p_r$ is a minimal
	idempotent in $A_{n-1}^F$ that $$E^{n}_{n-1}(p_t)=E^{n}_{n-1}(\tilde
		p_rp_t\tilde p_r)= \tilde p_rE^{n}_{n-1}(p_t)\tilde p_r=\alpha(t) \tilde p_r,$$
	for some $\alpha(t) \in F$. As $\tr(p_t)=\tr(E^{n}_{n-1}(p_t))=\alpha(t)
		\tr(\tilde p_r)$, we have $\alpha(t)= \tr(p_t)/\tr(\tilde p_r) =
		{d^F_{\la(t)}}/{\bm \delta d^F_\mu}$, using the induction hypothesis. It
	follows that $\bar e_np_t\bar e_n= \alpha(t) \bar e_n\tilde p_r.$ As $\bar
		e_n\tilde p_r$ is a minimal idempotent in $A_n^Fe_nA_n^F$, this implies that
	$(1/\alpha(t))p_t \bar e_np_t$ is a minimal idempotent. It also follows that
	$$(p_t\bar e_np_s)(p_s\bar e_n p_t)=\alpha(s) p_t\bar e_np_t,\quad (p_s\bar
		e_np_t)(p_t\bar e_n p_s)=\alpha(t) p_s\bar e_np_s.$$ The hypotheses of Lemma
	\ref{inductionargument} are satisfied with $$ p_\la = (1/\alpha(t)) p_t \bar
		e_n p_t, \ \ p_\nu = (1/\alpha(s)) p_s\bar e_n p_s, \ \ h = p_t \bar e_n p_s, \
		\ \text{and} \ \ h' = p_s \bar e_n p_t. $$ Applying Lemma
	\ref{inductionargument} gives a system of $K$-evaluable matrix units for the
	simple component $z_{(\mu, n+1)} A_{n+1}^F$, with these matrix units refining
	the given matrix units in $A_n^F$. The weight of the trace on the simple
	component $z_{(\mu, n+1)} A_{n+1}^F$ is $\tr(\bar e_n \tilde p_r) =
		(1/\bm\delta^2) \tr(\tilde p_r) = d^F_\mu/\bm\delta^{n+1}$, using Equation
	\ref{bimoduleproperty} and the induction hypothesis.

	Hence, we have proved the claim for matrix units $(f_{ts})$ in $A_n^Fe_nA_n^F$.
	In particular, the central idempotent $z_{n+1}=\sum_t f_{tt}$ of the ideal
	$A_n^Fe_nA_n^F$ is $K$-evaluable. Let $(h_{i,j})$ be a $K$-evaluable system of
	matrix units for $Q_{n+1}$; as $A_{n+1} \otimes_R R_{\mathfrak I} \to Q_{n+1}
		\otimes_R R_{\mathfrak I}$ is surjective, lift these to elements $(\tilde
		h_{i,j})$ in $A_{n+1} \otimes_R R_{\mathfrak I}$. Then $((1-z_{n+1}) \tilde
		h_{i,j})$ is a $K$-evaluable system of matrix units for
	$(1-z_{n+1})A_{n+1}^F\cong Q_{n+1}^F$.

	Thus we have obtained a complete system of $K$-evaluable matrix units in
	$A_{n+1}^F$.
\end{proof}

In the following statement, let $d^F_\gamma$ be as in Corollary \ref{corollary:
	generic semisimplicity}, so that the weights of the trace on $A_n^F$ are given
by $\omega_{(\gamma, n)} = d^F_\gamma/ \bm \delta^n$.

\begin{proposition}
	\label{traceevaluation} Adopt the hypotheses of Lemma \ref{Anmatrixunits}.
	\begin{enumerate}
		[label = {\normalfont (\arabic*)}]
		\item  If $x \in A_n \otimes _R R_\mathfrak I$, then $\tr(x)$ is $K$--evaluable, and
		      $\tr(x \otimes_{R_\mathfrak I} K) = \theta(\tr(x))$, where $\theta: R_\mathfrak I \to K$ is the evaluation homomorphism.
		\item  For $n \le m = \min\{n_0, n_1\}$, the quantities $\omega_{(\gamma, n)} =
			      d^F_\gamma/ \bm \delta^n$ and $d^F_\gamma$ are necessarily $K$-evaluable, and
		      the weights of the trace on $A_n^K$ are given by $(d_\gamma^K/(\delta^K)^n)$.
		\item 
			Either $m$ is equal to the the least $n \le n_1$ such that there exists $\gamma \in \Gamma_n$
		      with $d^F_\gamma$ $K$-evaluable and $d_\gamma^K = 0$,  or $m = n_1$ if no such $n$ exists.
	\end{enumerate}
\end{proposition}

\begin{proof}  We have already observed that (1) holds if  $x \in A_n$.   The statement for $x \in A_n \otimes _R R_\mathfrak I$ follows, by
writing $x$ as an $R_\mathfrak I$ linear combination of basis elements in $A_n$.   We have seen  in Lemma \ref{Anmatrixunits} that every
equivalence class of minimal idempotent in $A_n^F$  is represented by a $K$--evaluable idempotent;   and every equivalence class of minimal
idempotent in $A_n^K$ is represented by the evaluation $p^K = p \otimes_{R_\mathfrak I} K$ of a $K$--evaluable idempotent $p$  in $A_n^F$.   Point
(2) follows from this and point (1).

	For point (3), we consider the cases $n_1  <  n_0$ and $n_0 \le n_1$ separately.
	If $n_1 <  n_0$, then $m = n_1$, $A_n^K$ is semisimple for $n \le n_1$, and
	the weights of the trace on $A_n^K$ are non-zero for $n \le  n_1$. Using part (2),
	we see that $d^F_\gamma$ is $K$-evaluable and $d_\gamma^K \ne 0$ when $|\gamma|
	\le  n_1$.   If $n_0 < n_1$, then $m = n_0$, and similar considerations show that for all $n \le n_0$ and all $\gamma \in \Gamma_n$,   $d^F_\gamma$ is $K$-evaluable, and moreover $n_0$ is the least $n$ such that $d^K_\gamma = 0$ for some $\gamma \in \Gamma_n$.  	
\end{proof}

\section{Examples}\label{examples:sec}

In this section, we verify that the Brauer algebras, the BMW algebras, and the
$q$-Brauer algebras satisfy the axioms of section \ref{axioms}.
\subsection{Hecke algebras}  \label{subsection Hecke algebras}

\begin{definition}
	Let $S$ be an integral domain with invertible element $q$.  The Hecke algebra $H_n(S; q^2)$  is the unital $S$-algebra with generators $T_i$
($1 \le i \le n-1$)  satisfying  the braid relations $T_i T_{i+1} T_i = T_{i+1} T_i T_{i+1}$,  and the quadratic relation
	$(T_i -q) (T_i + q\inv) = 0$.
\end{definition}

The specialization at $q = 1$ of the Hecke algebra $H_n(S; q^2)$ is the
symmetric group algebra $S \mathfrak S_n$. For our primary examples, the Brauer
algebras, the BMW algebras, and the $q$-Brauer algebras, either the symmetric
group algebra or the Hecke algebra will play the role of $Q_n$.

There is an embedding of $H_n(S; q^2)$ into $H_{n+1}(S; q^2)$ determined by
$T_i \mapsto T_i$ for $1 \le i \le n-1$. The algebras $H_n(S; q^2)$ have a
consistent algebra involution $^*$ fixing the generators, $T_i^* =T_i$. There
is an automorphism $\#$ of $H_n(S; q^2)$ determined by $T_i^\# = - T_i\inv$.
For $w \in \mathfrak S_n$, let $w = s_{i_1} s_{i_2} \cdots s_{i_k}$ be a
minimum length expression. Then $T_w = T_{i_1} T_{i_2} \cdots T_{i_k}$ is well
defined and the set of $T_w$ for $w \in \mathfrak S_n$ is a basis of $H_n(S;
	q^2)$.

Dipper and James ~\cite{dipper-james1, dipper-james2} studied the
representation theory of the Hecke algebras, defining Specht modules $S^\la$
that generalize Specht modules for symmetric groups. Jost ~\cite{jost} showed
that restrictions of Specht modules from $H_n$ to $H_{n-1}$ have Specht
filtrations. Murphy ~\cite{murphy-hecke92} showed that the Hecke algebras are
cellular (before the formalization of the notion of cellularity in
~\cite{Graham-Lehrer}). The partially ordered set in the cell datum for $H_n$
is the set $\Lambda_n$ of Young diagrams of size $n$, with dominance order.

Murphy shows that his cell modules $\Delta^\la$ satisfy $\Delta^\la \cong
	({S^{\la'}})^\#$, where $\la'$ is the transpose of $\la$ and the superscript
$\#$ means that the module is twisted by the automorphism $\#$. Thus it follows
from the results of Dipper--James, Jost, and Murphy that the Hecke algebras
form a restriction coherent sequence of cellular algebras.

\begin{notation}
	\label{notation:  q integer and q-characteristic}    \phantom{xxx}
	\begin{enumerate}
		\item
		      For $d$ an integer,  the $q$ integer $\qint d q$  is defined to be $ \frac{q^d  -q^{-d}}{q - q\inv}$  if $q \ne \pm 1$,   and
$\qint d q = d$ or $(-1)^d d$ if $q = 1$ or $q=-1$ respectively.
		\item Let $K$ be a field and let $q \in K$. Let $e(q)$ be the smallest positive
		      integer $d$ such that $\qint d q = 0$; set $e(q) = \infty$ if $q^2$ is not a
		      proper root of unity.
	\end{enumerate}
\end{notation}

With this notation, the Hecke algebras $H_n(K; q^2)$ are split semisimple when
$n < e(q)$ and non-semisimple when $n \ge e(q)$. The branching diagram for the
sequence of Hecke algebras is Young's lattice. A convenient reference for these
facts is \cite{Mathas-book}.

The following is a restatement of the well-known fact that semisimple Hecke
algebras have a partition of unity by projections labeled by Young diagrams. It
was e.g.\ proved in \cite{Wenzl-type-A-subfactors}, Corollary 2.3 by a slightly
more complicated method.

\begin{proposition}
	Let $K$ be a field.  For $q \in K$ let $e(q)$  be defined as above.   Then the sequence of algebras
	$(H_n(q^2))_{n < e(q)}$  has a system of compatible $K$-evaluable matrix units.
\end{proposition}

\begin{proof}
	This is well-known, and can be fairly easily deduced from e.g. the exposition in \cite{Mathas-book}.  We sketch the proof.  Write $e$ for
$e(q)$.
	We regard standard tableaux as paths on Young's lattice.
	For a standard tableau $t$ of length  $n < e$,  let $F_t$ be the Murphy idempotent defined in  \cite{Mathas-book}, 3.33.   Then the elements
	$F_t$  for $t$ of fixed length $n < e$  are a $K$-evaluable partition of unity in $H_n(q^2)$.   Moreover,  if $s$ is a standard tableau of
length $k$ and $t$ is a standard tableau of length $n$ with $k \le n <e$,  then $F_s F_t = F_t$ if $s$ is a subtableau of $t$, and zero otherwise.
Hence the partitions of unity in $(H_n)_{n <e}$ are compatible.

	The existence of the $K$-evaluable off-diagonal matrix units can now be deduced
	using Lemma \ref{inductionargument}. Supposing $n < e$ and that compatible
	systems of $K$-evaluable matrix units exist in $H_k(q^2)$ for $k < n$, it
	suffices to consider tableaux $s$ and $t$ of length $n$ that are identical
	except that the boxes for $n-1$ and $n$ are switched; i.e. $s = (n-1, n) \cdot
		t$, where $(n-1, n)$ is the transposition in $\mathfrak S_n$. We apply the
	argument of Lemma \ref{inductionargument} with the idempotents $F_t$ and $F_s$
	and with $h$ and $h'$ both equal to $T_{n-1}$. It follows from Young's
	seminormal representations for the Hecke algebras, see \cite{Hoefsmit,
		Wenzl-type-A-subfactors} or \cite{dipper-james2}, Theorem 4.9, that the
	quantity $\alpha$ in our lemma is equal to $ \qint {d-1} q \qint {d+1} q/{\qint
		d q^2} $ where $d=i_n-i_{n+1}+j_{n+1}-j_n$ and where $(i_k,j_k)$ are row and
	column indices of the box containing $k$. Since $n < e$, $\alpha$ is
	$K$-evaluable and non-zero in $K$.
\end{proof}
\subsection{Brauer  algebras}

An $n$-Brauer diagram is a graph with $2n$ vertices, $n$ upper vertices and $n$
lower vertices, connected in pairs by $n$ edges. An edge that connects two
vertices on the same level is ``horizontal" and an edge that connects two
vertices on different levels is ``vertical". The product $ab$ of two $n$-Brauer
diagrams $a$ and $b$ is obtained by stacking $a$ over $b$, identifying the
lower vertices of $a$ with the upper vertices of $b$; the resulting graph may
have some number $k$ of loops, not connected to an upper or lower vertex. Let
$c$ be the $n$-Brauer diagram obtained by removing the loops. The product $a b$
is then defined to be $\delta^k c$, where $\delta$ is a variable or parameter.
We will insist that the ``loop parameter" $\delta$ be an invertible element in
some integral domain.

\begin{definition}
	Let $S$ be an integral domain with a distinguished invertible element $\delta$.   The $n$-Brauer algebra $\Br_n(S; \delta)$ is the free $S$
module with the basis of $n$-Brauer diagrams,  with bilinear multiplication determined by the multiplication of $n$-Brauer diagrams.
\end{definition}

By convention, $\Br_0(S; \delta) = S$. The generic ground ring for the Brauer
algebras is $R = \Z[\bm \delta\powerpm]$, where $\bm \delta$ is an
indeterminate. For any $S$, $\Br_n(S; \delta) = \Br_n(R; \bm \delta) \otimes_R
	S$. The field of fractions $F$ of $R$ is $\Q(\bm \delta)$.

When the context is clear, we may omit the ring $S$ or the parameter $\delta$
from the notation. Alternatively, we may use superscripts to indicate
specializations, for example $\Br_n^F(\bm \delta) = \Br_n(F; \bm \delta)$, as
in section \ref{section: axioms}. The same remark holds for the other examples
discussed below.

Let $\mathfrak S_n$ denote the symmetric group on $n$ objects. We will show
that the sequences of algebras $A_n = \Br_n(R; \bm \delta)$ and $Q_n = R
	\mathfrak S_n$ (with $R = \Z[\bm \delta\powerpm]$) satisfy the axioms of
section \ref{section: axioms}, with $\Gamma$ equal to Young's lattice.

Let $e_j$ and $s_j$ denote the $n$--Brauer diagrams: \newline \centerline{$ e_j
		= \inlinegraphic[scale=.7]{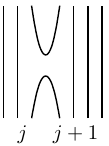}\qquad s_j = \inlinegraphic[scale=
			.7]{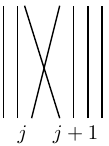} $} \newline It is easy to see that $e_1, \dots, e_{n-1}$ and
$s_1, \dots, s_{n-1}$ generate $\Br_n$ as an algebra.

Note that $e_j^2 = \delta e_j$; hence $\bar e_j = (1/\delta) e_j$ is an
idempotent. A diagrammatic computation shows that the elements $e_j$ satisfy
the Temperley-Lieb relations. The elements $s_1, \dots, s_{n-1}$ generate a
subalgebra isomorphic to the symmetric group algebra, namely the subalgebra
spanned by Brauer diagrams with only vertical strands.

There is a map $\iota$ from $n$-Brauer diagrams to $n+1$-Brauer diagrams, by
adding a pair of vertices and a vertical strand connecting them on the right.
\newline \centerline{$\iota: \quad \inlinegraphic{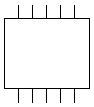} \quad \mapsto
		\quad \inlinegraphic{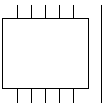} $} \newline The linear extension of $\iota$ to
$\Br_n$ is an injective algebra homomorphism into $\Br_{n+1}$. Using this, we
identify $\Br_n$ as a subalgebra of $\Br_{n+1}$.

There is an algebra involution $^*$ on $\Br_n$, which maps a Brauer diagram to
its reflection in a horizontal axis, interchanging upper and lower vertices.
The involution satisfies $\iota(a^*) = \iota(a)^*$. The generators $e_j$ and
$s_j$ satisfy $e_j^* = e_j$ and $s_j^* = s_j$.

For $n \ge 1$ define a map $\cl$ from $n$--Brauer diagrams into $\Br_{n-1}$ as
follows. First ``partially close" a given $n$--Brauer diagram by adding an
additional smooth curve connecting the rightmost upper and lower vertices,
\newline \centerline{$ \inlinegraphic{tangle_box2} \quad \mapsto \quad
		\inlinegraphic{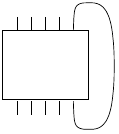}. $} \newline In case the resulting ``tangle"
contains a closed curve (which happens precisely when the original diagram
already had a vertical strand connecting the rightmost upper and lower
vertices), remove this loop and replace it with a factor of $\delta$. The
linear extension of $\cl$ to $\Br_n$ is a (non-unital) $\Br_{n-1}$--$\Br_{n-1}$
bimodule map, and $\cl\circ\,\iota(x) = \delta x$ for $x \in \Br_n$. Hence
$E^n_{n-1} (x) := (1/\delta) \cl(x)$ is a conditional expectation, and
$E^n_{n-1} (\iota(x)) = x$ for $x \in \Br_{n-1}$. A diagrammatic computation
shows that for $b \in B_n$, $e_n b e_n = \cl(b) e_n$, so $\bar e_n b \bar e_n =
	E^n_{n-1}(b) \bar e_n$.

The composition $\tr = E^1_0 \circ \cdots \circ E^n_{n-1} : \Br_n \to \Br_0
	\cong S$ can be computed as follows on an $n$-Brauer diagram $b$. ``Fully
close" the diagram by joining each upper vertex with the corresponding lower
vertex by a curve, and count the number $c$ of closed loops in the resulting
$0$-tangle. Then $\tr(b) = \delta^{c-n}$. It is not hard to check by pictures
that the full closure of $\chi_i b$ and the full closure of $b \chi_i$ are
homotopic, where $\chi_i$ stands for $e_i$ or $s_i$. It follows that $\tr$ is a
normalized trace.

The span $J_n$ of $n$-Brauer diagrams with at least one horizontal strand is an
ideal in $\Br_n(S; \delta)$, and is generated as an ideal by any one of the
elements $e_j$. The span of diagrams with only vertical strands is a subalgebra
isomorphic to the symmetric group algebra $S \mathfrak S_n$, and a linear
complement of $J_n$. Hence $B_n/J_n \cong S \mathfrak S_n$.

\subsubsection{Verification of the axioms}
We begin the verification of the axioms of section \ref{section: axioms}, for
$A_n = \Br_n(R; \bm \delta)$ and $Q_n = R \mathfrak S_n$, with $R = \Z[\bm
		\delta\powerpm]$, and with Young's lattice $\Lambda$ as $\Gamma$. Axioms
(\ref{axiom: involution on An}) and (\ref {axiom: A0 and A1}) are evident from
the discussion above. Axioms (\ref{axiom: restriction coherent quotients}) and
(\ref{axiom: generic ss quotients}) follow from the discussion of the Hecke
algebras in section \ref{subsection Hecke algebras}, since the symmetric group
algebras are specializations of the Hecke algebras. In particular, if $K$ is a
field of characteristic zero, then $Q_n^K$ is semisimple for all $n$, and if
$K$ has characteristic $p$, then $Q_n^K$ is semisimple if and only if $n \le
	p-1$. In the notation of Theorem \ref{first semisimplicity theorem - axiomatic
	version}, $n_1 = \infty$ in the first case and $n_1 = p-1$ in the second case.

The cellularity of the Brauer algebras (axiom (\ref{axiom: cellular Axn})) is
shown in ~\cite{Graham-Lehrer}. We will discuss the form of the cellular basis
below, in connection with axiom (\ref{axiom: weak coherence of An}). Axiom
(\ref{axiom: idempotent and Qn as quotient of An}) follows from the discussion
above and the explicit description of the cellular basis in Proposition
\ref{Graham-Lehrer cellularity theorem}, since $J_{n+1} = \Br_{n+1} e_{n}
	\Br_{n+1}$ and $\Br_{n+1}/J_{n+1} \cong R \mathfrak S_{n+1}$. We have already
mentioned that the elements $e_j$ satisfy the Temperley-Lieb relations, so
axiom (\ref{axiom: TL relations}) is satisfied. It is evident from the
multiplication of Brauer diagrams that $e_n$ commutes with $\Br_{n -1}$. Since
$e_n b e_n = \cl(b) e_n $ for $b \in \Br_n$, we have $e_n \Br_n e_n \subseteq
	\Br_{n-1} e_n$. Thus axiom (\ref{axiom: en An en}) is satisfied. Axiom
(\ref{axiom: An en}) is verified by a short argument using
~\cite{Wenzl-Brauer}, Proposition 2.1. See for example ~\cite{MR2794027}, Lemma
5.3. We have already observed that axiom (\ref{axiom: trace}) holds.

To show axiom (\ref{axiom: faithfulness of tr}), faithfulness of the trace on
$\Br_n^F$, where $F = \Q(\bm \delta)$, it suffices to show that $D =
	\det((\tr(b b')_{b, b'} )) \ne 0$, where $b, b'$ run through the diagram basis.
But for any $b, b'$, $\tr(b b') = \bm \delta^k$, where $k \le 0$, and $\tr(b
	b') = 1$ if and only if $b'= b^*$. Hence, $D = \pm 1 + \sum_{j < 0} n_j \bm
	\delta^j \ne 0$. Alternatively, use ~\cite{Wenzl-Brauer}, Theorem 3.2.

Axiom (\ref{axiom: weak coherence of An}) is verified below.

\subsubsection{A cellular basis of $\Br_n$}    We now describe the cellular basis of $\Br_n$ from ~\cite{Graham-Lehrer} and
verify axiom (\ref{axiom: weak coherence of An}).

\newcommand\sym{\mathfrak S}

For $1 \le a < b \le n$, let $\sym_{a, b} \subset \sym_n$ denote the
permutations of $\{a, a+1, \dots, b\}$.

\begin{definition}
	\label{definition:  shortest coset reps}  Let $1 \le s \le n/2$  and let $\ell = n - 2s$.
	Let  $\mathcal D_{n, s} $  denote the set of $\pi \in \sym_n$ such that
	$$
		\pi(\ell+ 1) < \pi(\ell+ 3) < \cdots < \pi(n -1),
	$$
	$$
		\pi(1) < \pi(2) < \cdots < \pi(\ell),
	$$
	and
	$$
		\pi(\ell+1) < \pi(\ell+2),  \pi(\ell+3) < \pi(\ell+4), \dots, \pi(n-1) < \pi(n).
	$$
\end{definition}

Note that for $s = 0$, $\mathcal D_{n, s}$ is the singleton set containing the
identity permutation. Recall that for $1 \le s \le n/2$, $\ex {n, s}$ denotes
$\ex {n, s} = e_{n - 2s +1} e_{n - 2s +3} \cdots e_{n-1} \quad (s \ \
	\text{factors})$. By convention $\ex {n, 0} = 1$.

\begin{lemma}
	\label{factorization of Brauer diagrams}   Every Brauer diagram $d$ on $n$ strands is either a permutation diagram, or, if $d$ has $\ell = n -
2s$ vertical strands for $1 \le s \le n/2$,  then $d$
	can be written uniquely as
	$$
		d =  u \, \pi \ex {n, s} \,v\inv,
	$$
	as a product in the Brauer algebra $\Br_n$,   where $u, v \in \mathcal D_{n,s}$ and  $\pi \in \sym_{\ell}$.
\end{lemma}

In the following, recall from Examples \ref{Young's lattice} and \ref{lattice
	for generic Brauer algebras} that $\Lambda_\ell$ denotes the set of Young
diagrams of size $\ell$, with dominance order $\unrhd$. Moreover, $\tilde
	\Lambda_n$ denotes the set of pairs $(\lambda, n)$, with $\lambda$ a Young
diagram of size $|\lambda| \le n$ and $n - |\lambda|$ even. The set $\tilde
	\Lambda_n$ is endowed with the partial order $(\lambda, n) \unrhd (\mu, n)$ if
$|\lambda| < |\mu|$ or if $|\lambda| = |\mu|$ and $\lambda \unrhd \mu$ in
dominance order. For a Young diagram $\lambda$, let $\mathcal T(\lambda)$
denote the set of standard Young tableaux of shape $\lambda$.

The following proposition says that one obtains a cellular basis of $\Br_n$ by
replacing the permutation $\pi \in \sym_\ell$ in Lemma \ref{factorization of
	Brauer diagrams} by elements of a cellular basis of $R \sym_\ell$.

\begin{proposition}
	{\normalfont (\cite{Graham-Lehrer}),   }     \label{Graham-Lehrer cellularity theorem}
	The Brauer algebra $\Br_n(R; \bm \delta)$ is cellular with cell datum described as follows. For each $\ell \le n$ with $n-\ell$ even,  choose
any cell datum for $R \sym_\ell$,  with partially ordered set $(\Lambda_\ell, \unrhd)$,
	index sets $\mathcal T(\lambda)$,  and cellular basis
	$\{m_{\mathfrak s, \mathfrak t}^\lambda:  \lambda \in \Lambda_\ell \ \text{and} \    \mathfrak s, \mathfrak t \in \mathcal T(\lambda)\}.
	$

	The partially ordered set in the cell datum for $\Br_n$ is $(\tilde\Lambda_n,
		\unrhd)$. For each $(\lambda, n) \in \tilde\Lambda_n$, where $\lambda$ has size
	$\ell = n-2s$, the index set $\mathcal T((\lambda, n))$ is the set of pairs
	$(\mathfrak s, u)$, where $\mathfrak s$ is a standard tableau of shape
	$\lambda$, and $u \in \mathcal D_{n, s}$, the family of permutations from
	Definition \ref{definition: shortest coset reps}.

	For $(\lambda, n) \in \tilde \Lambda_n$ and $(\mathfrak s, u), (\mathfrak t, v)
		\in \mathcal T((\lambda, n))$, let $$ x^{(\lambda, n)}_{(\mathfrak s, u),
				(\mathfrak t, v)} = u \,m^\lambda_{\mathfrak s, \mathfrak t} \ex {n, s} \,v\inv
	$$ Then the set of $ x^{(\lambda, n)}_{(\mathfrak s, u), (\mathfrak t, v)} $
	with $(\lambda, n) \in \tilde \Lambda_n$ and $(\mathfrak s, u), (\mathfrak t,
		v) \in \mathcal T((\lambda, n))$ is a cellular basis of $\Br_n$.
\end{proposition}

Now consider $k < n$ with $n-k$ even and let $s = (n-k)/2$. Moreover, let $\ell
	\le k$ with $k -\ell$ even and write $t = (k -\ell)/2$. Observe that $\mathcal
	D_{k, t} \subseteq \mathcal D_{n, s+t}$ and $\ex {k, t} \ex {n, s} = \ex {n,
		s+t}$. Let $\lambda \in \Lambda_\ell$ and consider $x^{(\lambda,
			k)}_{(\mathfrak s, u), (\mathfrak t, v)} = u\, m^\lambda_{\mfs, \mft} \ex {k,
		t} \,v\inv$, a cellular basis element in $\Br_k^{\unrhd (\lambda, k)} \setminus
	\Br_k^{\rhd (\lambda, k)}$. Then $$ x^{(\lambda, k)}_{(\mathfrak s, u),
			(\mathfrak t, v)} \ex {n, s} = u\, m^\lambda_{\mfs, \mft} \ex {n, s + t}
	\,v\inv = x^{(\lambda, n)}_{(\mathfrak s, u), (\mathfrak t, v)} \in
	\Br_n^{\unrhd (\lambda, n)} \setminus \Br_n^{\rhd (\lambda, n)}. $$ Hence,
axiom (\ref{axiom: weak coherence of An}) holds.

\subsection{BMW algebras}  The BMW algebras were first introduced in ~\cite{Birman-Wenzl} and
~\cite{Murakami-BMW} as abstract algebras defined by generators and relations.
The version of the presentation given here follows \cite{Morton-Traczyk} and
~\cite{Morton-Wassermann}.

\begin{definition}
	\label{definition: BMW algebra}
	Let $S$ be a commutative unital ring with invertible elements $r$, $q$, and $\delta$ satisfying $r\inv - r = (q\inv -q)(\delta -1)$.  The {\em
Birman--Wenzl--Murakami algebra}
	$\BMW_n(S; r, q, \delta)$ is the unital $S$--algebra
	with generators $g_i^{\pm 1}$  and
	$e_i$ ($1 \le i \le n-1$) and relations:
	\begin{enumerate}
		\item (Inverses) \hskip4pt $g_i g_i\inv = g_i\inv g_i = 1$.
		\item (Essential idempotent relation)\hskip4pt $e_i^2 = \delta e_i$.
		\item (Braid relations) \hskip4pt $g_i g_{i+1} g_i = g_{i+1} g_i g_{i+1}$
		      and $g_i g_j = g_j g_i$ if $|i-j|  \ge 2$.
		\item (Commutation relations)  \hskip4pt $g_i e_j = e_j g_i$  and
		      $e_i e_j = e_j e_i$  if $|i-j|\ge 2$.
		\item (Tangle relations)\hskip4pt $e_i e_{i\pm 1} e_i = e_i$, $g_i
			      g_{i\pm 1} e_i = e_{i\pm 1} e_i$, and $ e_i  g_{i\pm 1} g_i=   e_ie_{i\pm 1}$.
		\item (Kauffman skein relation)\hskip4pt  $g_i - g_i\inv = (q - q\inv)(1- e_i )$.
		\item (Untwisting relations)\hskip4pt $g_i e_i = e_i g_i = r\inv e_i$,
		      and $e_i g_{i \pm 1} e_i = r e_i$.
	\end{enumerate}
\end{definition}

The generic ground ring for the BMW algebras is $$ R = \Z[\bm r^{\pm1}, \bm
		q^{\pm1}, \bm \delta^{\pm1}]/\langle \bm r\inv - \bm r = (\bm q\inv - \bm
	q)(\bm \delta -1) \rangle, $$ where $\bm r$, $\bm q$, and $\bm \delta$ are
indeterminates over $\Z$. $R$ is an integral domain whose field of fractions is
$F \cong \Q(\bm r, \bm q)$ (with $\bm \delta = (\bm r\inv - \bm r)/(\bm q\inv -
	\bm q) + 1$ in $F$.)

The BMW algebra $\BMW_n$ can also be realized as the algebra of $n$--tangle
diagrams modulo regular isotopy and the following {\em Kauffman skein
		relations:}
\begin{enumerate}
	\item Crossing relation: $ \quad \inlinegraphic[scale=.6]{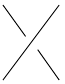} -
		      \inlinegraphic[scale=.3]{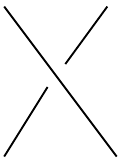} \quad = \quad (q\inv - q)\,\left(
		      \inlinegraphic[scale=1.2]{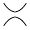} -
		      \inlinegraphic[scale=1.2]{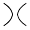}\right). $
	\item Untwisting relation: $\quad \inlinegraphic{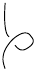} \quad = \quad r \quad
		      \inlinegraphic{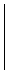} \quad\ \text{and} \quad\
		      \inlinegraphic{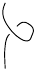} \quad = \quad r \inv \quad
		      \inlinegraphic{vertical_line}. $
	\item  Free loop relation: $T\, \cup \, \bigcirc = \delta \, T, $ where $T\, \cup \,
		      \bigcirc$ means the union of a tangle diagram $T$ and a closed loop having no
	      crossings with $T$.
\end{enumerate}
In the tangle picture,
$e_j$ and $g_j$ are represented by the following $n$--tangle diagrams:  \newline
\centerline{$
		e_j =  \inlinegraphic[scale=.7]{ordinary_E_j}\qquad
		g_j =  \inlinegraphic[scale= .7]{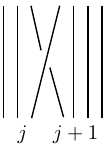}
	$}  \newline
The realization of the BMW algebra as an algebra of tangles is   from ~\cite{Morton-Wassermann},  where it is also shown that the BMW algebra is
free over the ground ring with a basis of certain tangles, and that the specialization of the BMW algebra at $q = 1$,  $r = 1$ is isomorphic to the
Brauer algebra.
We will freely identify the BMW algebras with the isomorphic tangle algebras.
We are going to show that the sequences of algebras $A_n = \BMW_n(R; \bm r, \bm q, \bm \delta)$  and $Q_n = H_n(R; \bm q^2)$  satisfy the axioms of
section \ref{section: axioms}.

There is an inclusion $\iota: \BMW_n \to \BMW_{n+1}$ by adding a vertical
strand on the right, as for the Brauer algebras. Due to the symmetry of the
relations, there is an algebra involution $*$ fixing the generators $e_i$ and
$g_i$; on tangles, the involution flips a tangle top to bottom. There is a map
$\cl: \BMW_n \to \BMW_{n-1}$ for $n \ge 1$ defined on tangles by partially
closing an $n$- tangle, as for Brauer diagrams, and reducing the resulting
$n-1$-tangle using the skein relations. One has $\cl \circ \iota (x) = \delta
	x$ for $x \in \BMW_{n-1}$ and hence $E^n_{n-1}(x) = (1/\delta) \cl(x)$ is a
conditional expectation from $\BMW_n$ to $\BMW_{n-1}$. By a tangle computation,
$e_n x e_n = \cl(x) e_n$ for $x \in \BMW_n$, and hence $\bar e_n x \bar e_n =
	E^n_{n-1}(x) \bar e_n$, where $\bar e_n$ is the idempotent $\bar e_n =
	(1/\delta) e_n$.

As for the Brauer algebras, the composition $\tr = E^1_0 \circ \cdots \circ
	E^n_{n-1} : \BMW_n \to \BMW_0 \cong S$ can be computed by ``fully closing" an
$n$-tangle, reducing the resulting $0$-tangle using the skein relations, and
dividing by $\delta^n$ to normalize the result. Since the full closure of
$\chi_i b$ and the full closure of $b \chi_i$ are homotopic, where $\chi_i$
represents $e_i$ or $g_i$, it follows that $\tr$ is a trace.

The ideal $J_n$ in $\BMW_n$ generated by the elements $e_i$, or any one of
them, satisfies $\BMW_n/J_n \cong H_n(S; q^2)$. In the tangle realization,
$J_n$ is the span of elements that factor as $x y$, where $x$ is an $(n,
	\ell)$-tangle and $y$ is an $(\ell, n)$ tangle, with $\ell < n$. There is a
section of the quotient map $\theta: \BMW_n \to H_n(S; q^2)$ defined as
follows: for $w \in \sym_n$, let $w = s_{i_1} s_{i_2} \cdots s_{i_k}$ be a
minimum length expression. Then $T_w = T_{i_1} T_{i_2} \cdots T_{i_k}$ and $g_w
	= g_{i_1} g_{i_2} \cdots g_{i_k}$ do not depend on the choice of the reduced
expression and $\theta(g_w) = T_w$. The span of $g_w$ for $w \in \sym_n$ is a
linear complement of $J_n$.

\subsubsection{Verification of the axioms.}
We proceed to verify the axioms of section \ref{section: axioms}, for $A_n =
	\BMW_n(R; \bm r, \bm q, \bm \delta)$ and $Q_n = H_n(R; \bm q^2)$, with $R$ the
generic ground ring, and with Young's lattice $\Lambda$ as $\Gamma$.

Axioms (\ref{axiom: involution on An}) and (\ref {axiom: A0 and A1}) are
evident. Axioms (\ref{axiom: restriction coherent quotients}) and (\ref{axiom:
	generic ss quotients}) follow from the discussion of the Hecke algebras in
section \ref{subsection Hecke algebras}. In particular, with the notation of
Theorem \ref{first semisimplicity theorem - axiomatic version} and Notation
\ref{notation: q integer and q-characteristic}, for a field $K$ and $q \in K$,
$n_1 = \infty$ if $q^2$ is not a proper root of unity and $n_1 = e(q) -1$
otherwise.

For axiom (\ref{axiom: cellular Axn}), we have the following theorem of Enyang \cite{Enyang-cellular-bases}.  (Cellularity was proved earlier by Xi \cite{Xi-BMW},  however only with respect to a total order refining $\unrhd$.)

\begin{theorem}
	{\normalfont (\cite{Enyang-cellular-bases}),   }  \label{Enyang's cellularity theorem}
	The {\normalfont BMW} algebra $\BMW_n(R; \bm r, \bm q, \bm \delta)$ is cellular with cell datum described as follows. For each $\ell \le n$
with $n-\ell$ even,  choose any cell datum for $H_\ell(R; \bm q^2)$,  with partially ordered set $(\Lambda_\ell, \unrhd)$,
	index sets $\mathcal T(\lambda)$,  and cellular basis
	$\{m_{\mathfrak s, \mathfrak t}^\lambda:  \lambda \in \Lambda_\ell \ \text{and} \    \mathfrak s, \mathfrak t \in \mathcal T(\lambda)\}.
	$    Also write $m_{\mathfrak s, \mathfrak t}^\lambda$ for the lift of this element in the span of $g_w$,  $w \in \sym_n$.

	The partially ordered set in the cell datum for $\BMW_n$ is $(\tilde\Lambda_n,
		\unrhd)$. For each $(\lambda, n) \in \tilde\Lambda_n$, where $\lambda$ has size
	$\ell = n-2s$, the index set $\mathcal T((\lambda, n))$ is the set of pairs
	$(\mathfrak s, u)$, where $\mathfrak s$ is a standard tableau of shape
	$\lambda$, and $u \in \mathcal D_{n, s}$, the family of permutations from
	Definition \ref{definition: shortest coset reps}.

	For $(\lambda, n) \in \tilde \Lambda_n$ and $(\mathfrak s, u), (\mathfrak t, v)
		\in \mathcal T((\lambda, n))$, let $$ x^{(\lambda, n)}_{(\mathfrak s, u),
				(\mathfrak t, v)} = g_u \,m^\lambda_{\mathfrak s, \mathfrak t} \ex {n, s} \,
		g_v^* $$ Then the set of $ x^{(\lambda, n)}_{(\mathfrak s, u), (\mathfrak t,
				v)} $ with $(\lambda, n) \in \tilde \Lambda_n$ and $(\mathfrak s, u),
		(\mathfrak t, v) \in \mathcal T((\lambda, n))$ is a cellular basis of $\BMW_n$.
\end{theorem}

Axiom (\ref{axiom: idempotent and Qn as quotient of An}) follows from the
discussion above, since $J_{n+1} = \BMW_{n+1} e_{n} \BMW_{n+1}$ and
$\BMW_{n+1}/J_{n+1} \cong H_{n+1}$. The Temperley-Lieb relations (axiom
(\ref{axiom: TL relations})) are included in the defining relations for
$\BMW_n$. It follows from the commutation relations in Definition
\ref{definition: BMW algebra} that $e_n$ commutes with $\BMW_{n -1}$. Since
$e_n b e_n = \cl(b) e_n $ for $b \in \BMW_n$, we have $e_n \BMW_n e_n \subseteq
	\BMW_{n-1} e_n$. Thus axiom (\ref{axiom: en An en}) is satisfied. Axiom
(\ref{axiom: An en}) is verified in the same manner as for the Brauer algebras,
using ~\cite{Birman-Wenzl}, Lemma 3.1. See for example ~\cite{MR2794027}, Lemma
5.10. We have already observed that axiom (\ref{axiom: trace}) holds.

Next we consider axiom (\ref{axiom: faithfulness of tr}), faithfulness of the
trace on $\BMW_n^F$, where $F$ is the field of fractions of the generic ground
ring $R$. Consider the ring homomorphisms from $R$ to $ \Z[\bm \delta
		\powerpm]$ sending $\bm r$ and $\bm q$ to $1$ and $\bm \delta $ to $\bm
	\delta$. We recall from ~\cite{Morton-Wassermann} that $\BMW_n^R \otimes_R
	\Z[\bm \delta \powerpm]$ is isomorphic to the Brauer algebra over $ \Z[\bm
		\delta \powerpm]$. Moreover, $\BMW_n^R$ has a basis $\{b_d\}$ of tangles
indexed by Brauer diagrams $d$ such that $b_d \otimes_R \Z[\bm \delta \powerpm]
	= d$ under the identification of $\BMW_n^R \otimes_R \Z[\bm \delta \powerpm]$
with the Brauer algebra. Furthermore, the trace on $\BMW_n$ specializes to the
trace of $\Br_n$. Now $$ \det((\tr(b_d b_{d'})_{d, d'})) \otimes_R \Z[\bm
		\delta \powerpm] = \det((\tr(d d')_{d, d'})) \ne 0, $$ This implies the
non-degeneracy of the trace on $\BMW_n^F$. Alternatively, one can use
~\cite{Wenzl-BCD}, Theorem 5.5.

Axiom (\ref{axiom: weak coherence of An}) is shown in the same manner as for
the Brauer algebras, using Theorem \ref{Enyang's cellularity theorem} in place
of Proposition \ref{Graham-Lehrer cellularity theorem}.

\subsection{$q$-Brauer algebras}   \label{subsection q Brauer}

\begin{definition}
	[\cite{Wenzl-qBr1}] \label{def qBr}  Let $S$ be an integral domain with invertible elements $q$,  $r$,   and $\delta$ such that
	$\delta (q - q\inv) = r - r\inv$.  For $n \ge 2$, the
	$q$-Brauer algebra $\qBr_n = \qBr_n(S; r, q, \delta)$ is the associative unital algebra with generators $e, g_1, \dots, g_{n-1}$, with
relations:
	\begin{enumerate}
		\item The elements $g_i$ are invertible and satisfy the type $A$ braid relations and
		      the quadratic (Hecke algebra) relation $(g_i - q)(g_i + q\inv) = 0$.
		\item $e^2 = \delta e$.
		\item $e g_i = g_i e$ for $i > 2$,   $e g_1 = g_1 e = q e$, and   $e g_2 e = r e$.
		\item  Define $e_{(2)} = e g_2 g_3 g_1\inv g_2\inv e$. Then $g_2 g_3 g_1 \inv g_2\inv
			      \ex 2 = \ex 2 =\ex 2 g_2 g_3 g_1\inv g_2\inv $.
	\end{enumerate}
\end{definition}

We put $\qBr_0 = \qBr_1 = S$ for convenience. Evidently, $g_1\inv e = e g_1\inv
	= q\inv e$. Using the quadratic relation for $g_2$, one has $e g_2\inv e =
	r\inv e$.

The generic ground ring for the $q$-Brauer algebras is $$ R = \Z[\bm r\powerpm,
		\bm q\powerpm, \bm \delta\powerpm ]/\langle \bm \delta (\bm q - \bm q\inv) =
	(\bm r - \bm r\inv)\rangle, $$ where $\bm r, \bm q$, and $\bm \delta$ are
indeterminates. That is, any instance of a $q$-Brauer algebra is a
specialization of the algebra over $R$, whose field of fractions is $F \cong
	\Q(\bm r, \bm q)$ (with $\bm \delta = (\bm r\inv - \bm r)/(\bm q\inv - \bm q)$
in $F$.)

It is shown in ~\cite{Wenzl-qBr1} that the specialization of the $q$-Brauer
algebra $\qBr_n$ at $q = 1$, $r = 1$ and $\delta$ arbitrary is isomorphic to
the Brauer algebra $\Br_n(\delta)$. According to ~\cite{Wenzl-qBr1}, the
subalgebra $H_n$ of $\qBr_n$ generated by the elements $g_j$ is isomorphic to
the Hecke algebra $H_n(q^2)$, as is the quotient by the ideal generated by $e$.

The $q$-Brauer algebras have natural embeddings $\qBr_n \hookrightarrow
	\qBr_{n+1}$, but they do not satisfy the axioms of section \ref{section:
	axioms} with the natural embeddings. Instead, for a fixed integer $M$, we will
define an increasing sequence of subalgebras $A_n$ of $\qBr_M(R; \bm r, \bm q,
	\bm \delta)$ such that $A_n \cong \qBr_n$, and the two sequences $A_n$ and $Q_n
	= H_n(R; \bm q^2)$ for $n \le M$ satisfy the axioms of section \ref{section:
	axioms}.

We will require a number of observations about the $q$-Brauer algebras that
were not available in the previous literature. We have placed the proof of
these observations in appendix \ref{appendix}. Unless specified, $\qBr_n$
designates the $q$-Brauer algebra over an arbitrary ground ring $S$, $\qBr_n =
	\qBr_n(S; r; q, \delta)$.

\begin{proposition}
	{\normalfont (~\cite{Nguyen1}, Proposition 3.14, see also appendix \ref{appendix: involution})}  \label{proposition involution}
	There is a unique algebra involution $*$ on $\qBr_n$  that fixes the generators, i.e. $e^* = e$, $g_i^* = g_i$ and $(g_i \inv)^* = g_i\inv$.
\end{proposition}

The elements $\ex k$ defined below play an essential role in the study of the
$q$-Brauer algebra in ~\cite{Wenzl-qBr1}. They specialize to the elements $e_1
	e_3 \cdots e_{2k-1}$ in the Brauer algebras. Note that the definition is
consistent with the previous definition of $\ex 2$.

\begin{notation}
	For $j \le k$, let  $g_{j, k}^+ = g_j g_{j+1}\cdots g_k$  and for $j \ge k$,  let $g_{j, k}^+ = g_j g_{j-1}\cdots g_k$.  Define elements
	$g_{j, k}^-$ similarly with $g_i$ replaced with $g_i\inv$.
\end{notation}

\begin{definition}
	\label{Definition ex k}  {\normalfont  (\cite{Wenzl-qBr1})}   $\ex 1 = e$  and for $k \ge 1$,
	$\ex {k+1} = e g_{2, 2k+1}^+ g_{1, 2k}^- \ex k$.
\end{definition}

\begin{lemma}
	{\normalfont (\cite{Nguyen1}, Proposition 3.14)}  \label{lemma e(k) * invariant}
	For $k \ge 1$,  $\ex k^* = \ex k$.
\end{lemma}

The sequence of elements $e_j$ defined below is new. They satisfy the
Temperley-Lieb relations and play a role similar to the contractions $e_j$ in
the Brauer or BMW algebras (although with a twist, as we shall see). For $j$
odd, they coincide with elements defined in ~\cite{rui2022jm}, and for $j \le
	3$ with elements defined in ~\cite{Wenzl-qBr1}.

\begin{definition}
	\label{q Brauer subalgebras A k}  \label{TL elements in q-Brauer}  In $\qBr_n$ for  $n \ge 3$, set
	$e_1 = e$, and for $1 \le j \le n-2$,
	$$
		e_{j+1} =
		\begin{cases}
			\Ad(g_j g_{j+1})(e_j)         & j  \ \text{is odd}   \\
			\Ad(g_j\inv g_{j+1}\inv)(e_j) & j  \ \text{is even}. \\
		\end{cases}
	$$
\end{definition}

The following statement is proved in appendix \ref{appendix TL}.

\begin{proposition}
	The elements $e_j$  satisfy the Temperley-Lieb relations $e_j^2 = \delta e_j$,  $e_j e_{j \pm 1} e_j = e_j$, and $e_j e_k = e_k e_j$ if $|j-k|
\ge 2$.   Moreover,
	$e_j$ commutes with $g_k$ for $k \ge j + 2$.
\end{proposition}

\begin{remark}
	In contrast with the Brauer algebras and the BMW algebras, the  elements $e_j$ do not satisfy $e_j = e_j^*$ for $j > 1$, and $e_j$  does not
commute with $g_k$ for $k \le j -2$.
\end{remark}

\begin{lemma}
	{\normalfont (\cite{rui2022jm}, Lemma 2.6 part (4)) }  \label{product expression for ex s} In $\qBr_n$ for $n \ge 3$,  for $1 \le k \le n/2$,
$$\ex k = e_1 e_3 \cdots e_{2k-1}.$$
\end{lemma}

\begin{notation}
	For $1 \le j < n$, define $\qBr_{j, n} = \alg\{e_j, g_j, g_{j+1}, \dots, g_{n-1}\}$, as a subalgebra of $\qBr_n$.
	$H_{j, n} = \alg\{g_j, \dots, g_{n-1}\}$  is the Hecke subalgebra of $\qBr_{j, n}$.
\end{notation}

For the remainder of this section, we fix some large natural number $M$ and
regard the algebras $\qBr_n$ for $n\le M$ as subalgebras of $\qBr_M$.

\begin{theorem}
	{\normalfont (\cite{Wenzl-qBr1}, Lemma 4.2, Theorem 4.5, and Lemma 5.2) }  \label{Theorem q-Brauer Markov trace}
	There is a unique unital  $S$ valued  trace $\tr$ on $\qBr_M = \qBr_M(S; r, q, \delta)$ with the ``Markov property":  $\tr(a g_1) = ( r/
\delta) \tr(a)$  and $\tr(a e) =  (1/ \delta) \tr(a)$
	for $a \in  \qBr_{2, M}$.
\end{theorem}

Of course, the restriction of $\tr$ to any $\qBr_n$ for $n \le M$ is the unique
Markov trace on $\qBr_n$. Moreover, the trace is consistent with
specializations, for example $\tr(x \otimes_R S) = \tr(x) \otimes_R S$ for $x
	\in \qBr_M(R; \bm r, \bm q, \bm \delta)$.

The following statement is proved in appendix \ref{appendix TL}.

\begin{lemma}
	\label{lemma: shift isomorphisms}    {\normalfont (The shift isomorphism)}
	For $j \ge 1$ and $k \ge 2$, with $j+k \le M$,    $e_1 \mapsto e_{j+1}$  and $g_\ell \mapsto g_{\ell +j }$  for $1 \le \ell \le k-1$ determines
an isomorphism $\psi_{j}$ from
	$\qBr_{k}$ to $\qBr_{j+1, j+k}$.   Moreover, the isomorphism is implemented by an inner automorphism of $\qBr_M$, and is therefore trace
preserving.
\end{lemma}

The isomorphism $\psi_j$ depends essentially only on $j$ and not on $k$, as
long as $j + k \le M$, in the sense that $\psi_j\circ \iota = \iota \circ
	\psi_j$, where $\iota$ denotes the embedding of $\qBr_k$ into $\qBr_{k+1}$ and
also the embedding of $\qBr_{j+k}$ onto $\qBr_{j+ k+1}$.

\begin{lemma}
	\label{lemma even shifts} $\psi_{2s} \circ \psi_k = \psi_{2s + k}$  for all $s$ and $k$.
\end{lemma}

\begin{proof}
	It suffices to show $\psi_2 \circ \psi_k = \psi_{2 + k}$,  since the general case follows by induction on $s$.  Moreover, it suffices to show
that $\psi_2 \circ \psi_k(e) = \psi_{2 + k}(e)$, that is,  $\psi_2(e_{k+1}) = e_{k+3}$.
	This statement follows by induction on $k$, using Definition \ref{TL elements in q-Brauer}.
\end{proof}

\subsubsection{Nguyen's cellular basis}   The theorem on cellular bases of the $q$-Brauer algebras below is due to Nguyen
in \cite{Nguyen2}. The statement here is more general than that in
\cite{Nguyen2}, but the idea is basically the same. We give a shorter proof in
the appendix \ref{appendix subsection: Nguyen}.

The following family of permutations $\mathcal D'_{n, s} $ replaces $\mathcal
	D_{n, s} $ from Definition \ref{definition: shortest coset reps} in the
description of the cellular basis of $\qBr_n$.

\begin{definition}
	\label{definition:  shortest coset reps q-Brauer}
	Let $1 \le s \le n/2$.  Let  $\mathcal D'_{n, s} $  denote the set of $\pi \in \sym_n$ such that
	$$
		\pi(1) < \pi(3) < \cdots  < \pi(2s -1),
	$$
	$$
		\pi(2s +1) < \pi(2s+2) < \cdots < \pi(n),
	$$
	and
	$$
		\pi(1) < \pi(2),  \pi(3) < \pi(4), \dots, \pi(2s-1) < \pi(2s).
	$$
\end{definition}

\begin{theorem}
	{\normalfont (\cite{Nguyen2}, Theorem 3.2)}    \label{Nguyen's theorem}     	The $q$-Brauer algebra $\qBr_n$ is cellular with cell datum described as
	follows: First, for each $k \le n$ with $n - k$ even, take some cell datum for
	$H_k(q)$, with partially ordered set $(\Lambda_k, \unrhd)$, index sets
	$\mathcal T(\lambda)$, the set of standard tableaux of shape $\lambda$, and
	cellular basis $\{b_{\mathfrak s, \mathfrak t}^\lambda: \lambda \in \Lambda_k \
		\text{and} \ \mathfrak s, \mathfrak t \in \mathcal T(\lambda)\}. $

	The partially ordered set in the cell datum for $\qBr_n$ is $(\tilde\Lambda_n,
		\unrhd)$, defined in Example \ref{lattice for generic Brauer algebras}. For
	each $(\lambda, n) \in \tilde\Lambda_n$, where $\lambda$ has size $k = n-2s$,
	the index set $\mathcal T((\lambda, n))$ is the set of pairs $(\mathfrak s,
		u)$, where $\mathfrak s$ is a standard tableau of shape $\lambda$, and $u \in
		\mathcal D'_{n, s}$, the family of permutations from Definition
	\ref{definition: shortest coset reps q-Brauer}.

	For $(\lambda, n) \in \tilde\Lambda_n$ and $(\mathfrak s, u), (\mathfrak t, v)
		\in \mathcal T((\lambda, n))$, let $$ x^{(\lambda, n)}_{(\mathfrak s, u),
				(\mathfrak t, v)} = g_u\ex s \psi_{2s}(b^\lambda_{\mathfrak s, \mathfrak t})
		g_v^* $$ Then the set of $ x^{(\lambda, n)}_{(\mathfrak s, u), (\mathfrak t,
				v)} $ with $(\lambda, n) \in \tilde\Lambda_n$ and $(\mathfrak s, u), (\mathfrak
		t, v) \in \mathcal T((\lambda, n))$ is a cellular basis.
\end{theorem}

The proof of the following corollary is similar to the verification of axiom
(\ref{axiom: weak coherence of An}) for the Brauer or BMW algebras.

\begin{corollary}
	\label{corollary: qBr weak coherence 1}
	Suppose $k < n$ with $n - k$ even and let $s = (n-k)/2$.   Let $(\gamma, k) \in \tilde \Lambda_k$.
	If $x \in \qBr^{\unrhd (\gamma, k)}_k \setminus \qBr^{ \rhd(\gamma, k)}_k$, then  $\ex s \psi_{2s}(x) \in
		\qBr^{\unrhd (\gamma, n)}_n \setminus \qBr^{ \rhd(\gamma, n)}_n$.
\end{corollary}

\subsubsection{The subalgebras $A_n$}
For $n \ge 1$ define $$A_n = \qBr_{M - n + 1, M} = \alg\{ e_{M-n+1}, g_{M-n+1},
	\dots, g_{M-1}\} = \psi_{M-n}(\qBr_n). $$ and $A_0 = S$. Evidently, the
sequence of algebras $A_n$ is increasing, $A_n \subset A_{n+1}$, and $A_M =
	\qBr_M$. The essential idempotents that play the role of the elements $e_i$ in
the axiom of section \ref{axioms} is the reversed sequence of the elements
$e_i$ of this section, $f_i = e_{M-i}$ for $1 \le i \le M-1$. We have $$ \{f_1,
	f_2, \dots, f_{n-1}\} \subset A_n. $$

The algebras $A_n$ are not invariant under the involution $*$ of $\qBr_M$.
However, each $A_n$ has its own involution induced by the isomorphism
$\psi_{M-n} : \qBr_n \to A_n$. Moreover, each $A_n$ is cellular with the
cellular basis $\psi_{M-n}(\mathcal C)$, where $\mathcal C$ is the Nguyen
cellular basis of $\qBr_n$.

\subsubsection{The left-handed alternative}
The algebras $\qBr_n(S; r, q, \delta )$ are also generated by \break $e,
	g_1\inv, \dots, g_{n-1}\inv$, and these generators satisfy the same relations
as in Definition \ref{def qBr}, but with $q$ replaced by $q\inv$ and $r$
replaced with $r\inv$. The loop parameter $\delta$ is unchanged, and $\ex 2$
remains unchanged, using ~\cite{Wenzl-qBr1}, Lemma 3.2. The involution defined
in terms of these generators coincides with the involution in Proposition
\ref{proposition involution}.

Using this, we define ``left-handed" versions $e_i'$ of the Temperley-Lieb
elements $e_i$ and left-handed versions $\exprime i$ of the elements $\ex i$;
the definitions are the same, except $g_i$ and $g_i\inv$ are interchanged. We
have $\exprime k = e'_1 e'_3 \cdots e'_{2k -1}$ by Lemma \ref{product
	expression for ex s}. The left-handed version of $\qBr_{k, n}$ is $\qBr'_{k, n}
	= \alg\{e'_k, g_k, \dots, g_{n-1}\}$.

\begin{lemma}
	\label{left hand lemma 1}  \phantom{xxxx}
	\begin{enumerate}
		[label = {\normalfont (\arabic*)}]
		\item  For all $j$, $e'_j = e_j^*$. \smallskip
		\item  For all $j$, $ \exprime j = \ex j.$
		\item  $\qBr'_{k, n} = (\qBr_{k,n})^*$.
	\end{enumerate}
\end{lemma}

\begin{proof}
	Point (1) is proved by induction on $j$.  Point (2) follows from point (1), together with Lemmas \ref{lemma e(k) * invariant} and \ref{product
expression for ex s}, and their left-handed analogues.  Point (3) is evident from point (1).
\end{proof}

\begin{lemma}
	\label{lemma shift of e(s)}    For all $k$ and $s$,
	$$
		\psi_k(\ex s) =  e_{k +1} e_{k+3} \cdots e_{k + 2s -1} = \ex {k+2s,s}.
	$$
\end{lemma}

\begin{proof}
	If $k$ is even, $\psi_k(e_j) = e_{j+k}$ for all $j$, by Lemma \ref{lemma even shifts}.   If $k$ is odd, then $\psi_k(e'_j) = e_{j+k}$ for all
$j$, by induction on $j$.  Hence the result follows from
	Lemma \ref{left hand lemma 1}.
\end{proof}

\begin{lemma}
	\label{q-Brauer left Markov property}
	The Markov trace $\tr$  of Theorem \ref{Theorem q-Brauer Markov trace}  satisfies
	$\tr(a e) =  (1/ \delta) \tr(a)$
	for $a \in  \qBr'_{2, M}$.
\end{lemma}

\begin{proof}
	Simply apply Theorem  \ref{Theorem q-Brauer Markov trace} to the isomorphic algebra $\qBr(S; r\inv, q\inv, \delta)$ with generators
	$e, g_1\inv, \dots g_{M-1}\inv$.
\end{proof}

\begin{lemma}
	\label{lemma q-Brauer Markov property}  For any $k$,  $e_k \in \qBr_{k, M}$ satisfies
	$$
		\tr(e_k a) = (1/\delta) \tr(a)
	$$
	for all $a \in \qBr_{k+1, M}$.
\end{lemma}

\begin{proof}
	For any $k$, $\psi_{k-1}(e) = e_k$.   If $k$ is odd, $e_{k+1} = \psi_{k-1}(e_2)$,  so $\qBr_{k +1 , M} =
		\psi_{k-1}(\qBr_{2, M-k+1})$, and the assertion follows from Theorem  \ref{Theorem q-Brauer Markov trace}.   If $k$ is even,  $e_{k+1} =
\psi_{k-1}(e'_2)$, so  $\qBr_{k + 1, M} =
		\psi_{k-1}(\qBr'_{2, M-k+1})$, and the assertion follows from Lemma   \ref{q-Brauer left Markov property}.
\end{proof}

\subsubsection{Verification of the axioms}
We are going to verify the axioms of section \ref{section: axioms}, for $A_n$
as above, defined over the generic ground ring $R$, and $Q_n = H_n(R; \bm
	q^2)$, for $0 \le n \le M$. Young's lattice $\Lambda$ (truncated at level $M$)
plays the role of $\Gamma$, and the elements $f_i$ ($1 \le i \le M-1$) the role
of the essential idempotents in section \ref{section: axioms}.

Axioms (\ref{axiom: involution on An}) and (\ref {axiom: A0 and A1}) are
evident. Axioms (\ref{axiom: restriction coherent quotients}) and (\ref{axiom:
	generic ss quotients}) follow from the discussion of the Hecke algebras in
section \ref{subsection Hecke algebras}. As for the BMW algebras, for $K$ a
field and $q \in K$, we have $n_1 = \infty$ if $q^2$ is not a proper root of
unity, and $n_1 = e(q) -1$ otherwise.

Axiom (\ref{axiom: cellular Axn}) follows from Theorem \ref{Nguyen's theorem}
and the isomorphism $\psi_{M-n}: \qBr_n \to A_n$.

Because of the isomorphism $\psi_{M-n-1}: \qBr_{n+1} \to A_{n+1}$ taking $e$ to
$f_{n}$, for axiom (\ref{axiom: idempotent and Qn as quotient of An}), it
suffices to verify the corresponding statements with $A_{n+1}$ replaced by
$\qBr_{n+1}$ and $f_n$ by $e$. It follows from the form of Nguyen's cellular
basis that the ideal in $\qBr_{n+1}$ generated by $e$ coincides with the span
of the basis elements corresponding to $(\lambda, n+1) \in \tilde
	\Lambda_{n+1}$ with $|\lambda| < n+1$. The isomorphism $\qBr_{n+1}/\langle e
	\rangle \ \cong \ H_{n+1}(\bm q^2)$ is from ~\cite{Wenzl-qBr1}.

We already know that the elements $f_i$ satisfy the Temperley-Lieb relations,
which is axiom (\ref{axiom: TL relations}). Axioms (\ref {axiom: en An en}) and
(\ref{axiom: An en}) are verified in appendix \ref{appendix TL}.

As we already know that $\tr$ is a trace, the content of axiom (\ref{axiom:
	trace}) is that the conditional expectations $E^k_{k-1}$ are trace-preserving.
But by Lemma \ref{lemma: markov trace}, this is equivalent to the Markov
property $\tr(f_{n-1} a) = (1/\delta) \tr(a)$ for $a \in A_{n-1}$. The
equivalence was proved in Lemma \ref{lemma: markov trace}; the verification of
the Markov property is in Lemma \ref{lemma q-Brauer Markov property}.

To verify axiom (\ref {axiom: faithfulness of tr}), it suffices to show that
$\tr$ is non-degenerate on $\qBr_n(F; \bm r, \bm q, \bm \delta)$ for all $n$,
where $F$ is the field of fractions of the generic ground ring $R$ (because the
isomorphism of $A_n$ with $\qBr_n$ is trace-preserving). According to
~\cite{Wenzl-qBr1}, Theorem 3.8, $\qBr_n(R)$ has a basis $(g_d)$, labelled by
Brauer diagrams on $n$ strands, with the property that, for each Brauer diagram
$d$, $g_d \otimes_R \Z[\bm \delta^\pm] = d$. It suffices to show that
$\det((\tr(g_d g_{d'})_{d, d'}) \ne 0$, where $d, d'$ range through Brauer
diagrams on $n$ strands. But this is a statement about the trace on
$\qBr_n(R)$. The statement follows from the non-degeneracy of the trace on the
generic Brauer algebra $\Br_n(\Z[\bm \delta \powerpm]; \bm \delta)\cong
	\qBr_n(\Z[\bm \delta \powerpm]; 1, 1, \bm \delta)$. Namely, if $\varphi : R \to
	\Z[\bm \delta \powerpm]$ is the ring homomorphism taking $\bm r \mapsto 1$ and
$\bm q \mapsto 1$, then $\varphi\circ \tr(x) = \tr( x \otimes_{R} \Z[\bm \delta
		\powerpm])$. Therefore, $\varphi(\det((\tr(g_d g_{d'})_{d, d'})) ) =
	\det((\tr(d d')_{d, d'})) \ne 0$.

For axiom (\ref{axiom: weak coherence of An}), consider $k < n$ with $n-k$ even
and let $s = (n-k)/2$. Moreover, let $\ell \le k$ with $k -\ell$ even and write
$t = (k - \ell)/2$. Let $\lambda \in \Lambda_\ell$ and consider a basis element
$$ x = x^{(\lambda, k)}_{(\mathfrak s, u), (\mathfrak t, v)} = g_u\ex t
	\psi_{2t}(b^\lambda_{\mathfrak s, \mathfrak t}) g_v^* $$ of $\qBr_k^{\unrhd
		(\la, k)} \setminus \qBr_k^{\rhd (\la, k)} $, where $u, v \in \mathcal
	D'_{k,\ell}$. Then $\psi_{M-k}(x)$ is a basis element of $A_k^{\unrhd (\la, k)}
	\setminus A_k^{\rhd (\la, k)} $. We have to show that $$ f_{(n, s)}
	\psi_{M-k}(x) = e_{(M-k, s)} \psi_{M-k}(x) \in A_n^{\unrhd (\la, n)} \setminus
	A_n^{\rhd (\la, n)} . $$
One has $\ex s \psi_{2s}(x) \in \qBr_n^{\unrhd (\la, n)} \setminus \qBr_n^{\rhd
		(\la, n)}$, by Corollary \ref{corollary: qBr weak coherence 1}, so
$\psi_{M-n}(\ex s \psi_{2s}(x)) \in A_n^{\unrhd (\la, n)} \setminus A_n^{\rhd
			(\la, n)}$. But one computes that $\psi_{M-n}(\ex s \psi_{2s}(x)) = e_{(M-k,
			s)} \psi_{M-k}(x) $, making use of Lemma \ref{lemma shift of e(s)}.
\section{Explicit semisimplicity criteria}\label{explicit:sec}

We have already checked all the axioms for the Brauer, $q$-Brauer and BMW
algebras, and we have determined the number $n_1$ for these algebras. In all of
these examples, the quantities $d_\ga^F$ (see Proposition
\ref{traceevaluation}) have been calculated explicitly in the literature, and
we can hence use that proposition to calculate $n_0$ and $m$ for any
specialization over a field $K$. As our index set is given by Young diagrams,
we will write $d_\la^F$ instead of $d_\ga^F$.

\subsection{Weights} Let $\la$ be a Young diagram. We denote its box in the $i$-th row and $j$-th
column by $(i,j)$. The quantity $h_\la(i,j)$ is the length of the hook in $\la$
with corner at $(i,j)$, given explicitly by
\begin{equation}
	h_\la(i,j)=\lambda_i-i+\lambda_j' -j+1,
\end{equation}
where $\la_i$ and $\la_j'$ denote the number of boxes in the $i$-th
row and $j$-th column of $\la$.
We will also need the following quantities
for a given Young diagram $\la$
\begin{equation}
	\label{eqn:  defn of d(i, j)}
	d_\la(i,j)\ =
	\begin{cases}
		\lambda_i+\lambda_j-i-j      & \text{if $i\leq j$,} \\
		-\lambda_i'-\lambda_j'+i+j-2 & \text{if $i> j$.}
	\end{cases}
\end{equation}
The numbers $d_\la(i,j)$ in the first few boxes of a sufficiently large Young diagram $\la$
can be visualized as follows
\[
	\begin{tabular}
		{|c|c|c|c|c|}
		\hline
		$2\la_1-2$         & $\la_1+\la_2-3$    & $ \la_1+\la_3-4$   & $\la_1+\la_4-5$ & $\la_1+\la_5-6$ \\
		\hline
		$-\la_1'-\la_2'+1$ & $2\la_2-4$         & $ \la_2+\la_3-5$   & $\la_2+\la_4-6$ & $\la_2+\la_5-7$ \\
		\hline
		$-\la_1'-\la_3'+2$ & $-\la_2'-\la_3'+3$ &
		$2\la_3-6$         & $\la_3+\la_4-7$    & ...                                                    \\
		\hline
		$-\la_1'-\la_4'+3$ & $-\la_2'-\la_4'+4$ & $-\la_3'-\la_4'+5$ & $2\la_4-8$      & ...             \\
		\hline
		$-\la_1'-\la_5'+4$ & $-\la_2'-\la_5'+5$ & ...                & ...             & ...             \\
		\hline
	\end{tabular}
\]
\begin{remark}
	\label{tabmonotone} It is straightforward to see that

	(a) the numbers $d_\la(i,j)$ decrease if we move to the right or down, as long as $i\leq j$, and

	(b) the numbers $d_\la(i,j)$ increase if we move to the right or down, as long as $i> j$.
\end{remark}

\subsection{Brauer algebras}

\begin{theorem}
	{\normalfont (\cite{Wenzl-Brauer})}   \label{theorem: Brauer trace formula}  The quantities $d_\la^F$ for the Brauer algebra $\Br_n^F(\bm
\delta)$, defined over
	the field $F=\Q(\bm \delta)$ are given by
	\begin{equation}
		\label{equation:  Brauer trace formula}
		d_\la^F(\bm \delta)=\prod_{(i,j)\in\la} \frac{\bm \delta+d_\la(i,j)}{h_\la(i,j)},
	\end{equation}
	where $\la$ is a Young diagram with $|\la|\in \{ n-2k, 0\leq k\leq \lfloor n/2\rfloor\}$.
\end{theorem}

Let $\delta$ be a non--zero integer. We want to find conditions for $\delta +
	d_\la(i, j) \in \Q$ to be zero.

\begin{notation}
	\label{notation:  quantities for vanishing weights}  Let $\delta$ be an integer.
	\begin{enumerate}
		\item
		      Let $m_0(\delta)$ be the least positive integer $n \ge2 $ such that there exists a Young diagram $\la$ of size $n$ and $(i, j) \in
\la$ such that  $\delta + d_\la(i, j) = 0$.
		\item  Let $m_1(\delta)$ be the least positive integer $n \ge2 $ such that there
		      exists a Young diagram $\la$ of size $n$ and $(i, j) \in \la$ with $i \ne j$
		      such that $\delta + d_\la(i, j) = 0$.
		\item Let $m_2(\delta)$ be the least positive integer $n \ge2 $ such that there
		      exists a Young diagram $\la$ of size $n$ and $i$ such that $(i, i) \in \la$ and
		      $\delta + d_\la(i, i) = 0$.
	\end{enumerate}
\end{notation}

Evidently $m_0(\delta) = \min \{m_1(\delta), m_2(\delta)\}$.

\begin{remark}
	\label{remark:  m sub i of zero}  $m_0(0) = m_2(0) = 2$,   and $m_1(0) = 3$,   as follows from direct examination of Young diagrams of size
$\le 3$.
\end{remark}

\begin{lemma}
	\label{lemma:  n zero prime for Brauer}   Let $\delta$ be a non-zero integer.
	\begin{enumerate}
		[label = {\normalfont (\arabic*)}]
		\item $m_1(\delta) =  - \delta +3$ if $\delta$ is negative and $m_1(\delta) = \delta + 1$ if $\delta$ is positive.
		\item $m_2(\delta) = -\delta/2 + 1$ if $\delta$ is negative and even and $m_2(\delta) = \infty$ otherwise.
		\item  $$m_0(\delta) =  \min \{m_1(\delta)\ , m_2(\delta)\}  =
			      \begin{cases}
				      \delta + 1      & \text{if} \  \delta  \ \text{is positive}          \\
				      -\delta/2 + 1   & \text{if} \  \delta \ \text{is negative and even}  \\
				      -\delta + 3  \  & \text{if}  \  \delta \  \text{is negative and odd} \\
			      \end{cases}
		      $$
	\end{enumerate}
\end{lemma}

\begin{proof}
	If $\delta $ is negative,  consider the one--row diagram $\la = (-\delta + 3)$.   Then $\delta +  d_\la(1, 2) = 0$,  so $m_1(\delta) \le
-\delta + 3$.  If $\mu$ is a Young diagram of size $|\mu| < -\delta +3$,  one can show that for all $(i, j)$ with $i \ne j$,  $\delta + d_\mu(i,j)
< 0$,  so $m_1(\delta) = -\delta + 3$.  Similarly, if $\delta$ is positive, consider the one--column diagram $\la = (1^{\delta + 1})$.  Then
$\delta +  d_\la(2, 1) = 0$, so $m_1(\delta) \le \delta + 1$.  If $\mu$ is a Young diagram of size $|\mu| < \delta +1$,  one can show that for all
$(i, j)$ with $i \ne j$,  $\delta + d_\mu(i,j) >  0$,  so $m_1(\delta) = \delta + 1$.   Note that $d_\la(i, i) = 2 \la_i - 2i \ge 0$, so $\delta +
d_\la(i, i) = 0$ requires that $\delta$ is negative and even.
	If $\delta$ is negative and even, let $\la = (-\delta/2 + 1)$.  Then $\delta + d_\la(1,1) = 0$, so $m_2(\delta) \le -\delta/2 + 1$.   It's easy
to check that if $|\mu| < -\delta/2 +1$,  then $\delta + d_\mu(i, i) <0$ for all $i$.  Hence $m_2(\delta) = -\delta/2 + 1$.   The expression for
$m_0(\delta)$ follows immediately.
\end{proof}

\begin{theorem}
	\label{Brauersolution}   Let $K$ be a field and $\delta$ a non-zero element of $K$.
	\begin{enumerate}
		[label = {\normalfont (\arabic*)}]
		\item If $\character K = 0$ and $\delta \not \in \Z $ then the Brauer algebra
		      $\Br_n^K(\delta)$ is semisimple for all $n\in \N$.
		\item If $\character K = p >0$ and $\delta \not \in \Z 1_K$, then the Brauer algebra
		      $\Br_n^K(\delta)$ is semisimple if and only if $n \le p-1$.
		\item If $\character K = 0$, and $\delta$ is an integer, then $\Br_n^K(\delta)$ is
		      semisimple if and only if $n \le m_0(\delta)$.
		\item Suppose $\character K = p >0$ and $\delta \in \Z 1_K$. Write $\delta = N 1_K$,
		      where $0 < N \le p-1$. Then $\Br_n^K(\delta)$ is semisimple if and only if $n
			      \le \min \{ p-1, m_0(N), m_0(N-p)\}$.
	\end{enumerate}
\end{theorem}

\begin{proof}
	Point (1) was already proved in  \cite{Wenzl-Brauer}. If $\character K = 0$,  then $n_1 = \infty$, and if $\character K = p$,  then $n_1 =
p-1$.  In both cases, if $\la$ is a Young diagram with $|\la| < n_1$,  then $d^F_\la$ is $K$  evaluable, since the hook lengths in the denominator
in (\ref{equation:  Brauer trace formula}) are non-zero in $K$.   It follows from Proposition \ref{traceevaluation}  that $\Br_n^K$ is semisimple
if and only if $n \le m$,   where $m$ is the least $n < n_1$ such that $d^K_\la = 0$ for some $\la$ with $|\la| = n$,  or $m = n_1$ if no such $n$ exists.    If $\delta \not \in \Z 1_K$,  then $d^K_\la$ is never zero,  which proves
(1) and (2).   Point (3) is immediate from Lemma \ref{lemma:  n zero prime for Brauer}.  For point (4),  $\delta + d_\la(i,j) = 0$ in $K$ if and
only if $N + d_\la(i,j) $  is a multiple of $p$.  The only possibilities are $N + d_\la(i,j) = 0$ and $N + d_\la(i,j)  = p$.  The latter is
equivalent to $N - p + d_\la(i,j) = 0$.   It follows that
	$$
		m =  \min \{ p-1, m_0(N), m_0(N-p)\}.
	$$
\end{proof}

\begin{remark}
	This recovers the result of \cite{Rui-Brauer-semisimplicity.2},  Corollary 2.5, part (1).   Rui and Si  also determine when $\Br^K_n(0)$ is
semisimple.
\end{remark}

\subsection{$q$-Brauer algebras} Observe that we can express $\bm \delta$ in terms of $\bm q$ and $\bm r$ in the
quotient field $F$ for our generic ring $R$. Hence we will only give the
formulas for $d_\la^F$ in terms of $\bm r$ and $\bm q$.
\begin{theorem}
	{\normalfont (\cite{Wenzl-qBr2}, Theorem 2.8) } The quantities $d_\la^F$  for the  $q$-Brauer algebra $\qBr_n^F(\bm  r,\bm q)$ with $F=\Q(\bm
r, \bm q)$
	are given by
	$$d_\la^F(\bm r, \bm q)=\prod_{(i,j)\in\la} \frac{\bm r\bm q^{d_\la(i,j)}-\bm r^{-1}\bm q^{-d_\la(i,j)}}{\bm q^{h_\la(i,j)}-\bm
q^{-h_\la(i,j)}}.$$
\end{theorem}

\begin{theorem}
	\label{qBrauersolution} If $K$ is a field in that $q=\pm 1$ or $r = \pm 1$, then $\qBr^K_n(r,q,\delta)\cong \Br^K_n(\delta)$. Hence it suffices
to consider fields $K$ where $q\neq\pm 1$, $r \neq \pm 1$ in the following.
	\begin{enumerate}
		[label = {\normalfont (\arabic*)}]

		\item The $q$-Brauer algebra $\qBr_n^K(r,q,\delta)$ is semisimple for all $n\in \N$
		      if $q$ is not a root of unity and $r\neq \pm q^N$ for any $N\in \Z$.

		\item   If $q^2$ is a primitive $e$-th root of unity but $r\neq \pm q^N$ for any $N\in
			      \Z$, then the $q$-Brauer algebra $\qBr_n^K(r,q,\delta)$ is semisimple if and
		      only if $n < e$.

		\item If $q$ is not a root of unity and $r= \pm q^N$ for some integer $N$, the
		      algebra $\qBr_n^K(r,q,\delta)$ is semisimple if and only if $n \le m_0(N)$, see
		      Lemma \ref{lemma: n zero prime for Brauer}.
		\item  Suppose $q^2$ is a primitive $e$-th root of unity, and $r=\pm q^N$. We may
		      assume $-e < N < 0$. Then the $q$-Brauer algebra $\qBr_n^K(r,q,\delta)$, is
		      semisimple if and only if $n \le \min\{e-1, m_0(N), m_0(N-e), m_0(e+N)\}$.
	\end{enumerate}
\end{theorem}

\begin{proof}
	Note that $n_1 = \infty$ if $q$ is not a root of unity, and $n_1 = e-1$ if $q^2$ is a primitive $e$-th root of unity,  since the Hecke algebras
$H_n(K; q^2)$  play the role of $Q^K_n$.

	Part (1) follows from this and the fact that each factor of $d_\la^K$ is
	well-defined and nonzero for all Young diagrams $\la$ if $q$ is not a root of
	unity and $r\neq\pm q^n$ for some $n\in \Z$. This was already shown in
	\cite{Wenzl-qBr1}, Theorem 5.3. Part (2) is similar.

	If $q$ is not a root of unity, and $r = \pm q^N$ for some integer $N$, observe
	that $d_\la^K=0$ if and only if $q^{2 N + 2d_\la(i,j)}= 1$ for some
	$(i,j)\in\la$, which is equivalent to $N+ d_\la(i, j) = 0$. Hence, the
	determination of $m =  n_0$ is exactly the same as for the Brauer algebras
	$\Br_n^\Q(N)$, in characteristic $0$, see Theorem \ref{Brauersolution}.

	Suppose $q^2$ is a primitive $e$-th root of unity and $r=\pm q^N$, for some
	integer $N$. We have $q^e = \pm 1$, so $r = \pm q^{N + k e}$ for all integers
	$k$. Since we exclude $r = \pm 1$, we can assume $e \nmid N$. Hence we can
	assume $-e < N <0$. We have $d_\la^K = 0$ if and only if $e \mid N + d_\la(i,
		j)$ for some $(i , j) \in \la$, that is $N - ke + d_\la(i,j) = 0$ for some
	integer $k$. Consequently, $$ m = \min \{ e-1, \min_k m_0(N - ke) \} = \min \{
		e-1, m_0(N-e), m_0(N), m_0(N+e) \}, $$ since other values of $k$ produce
	superfluous bounds, %
\end{proof}

\begin{remark}
	One can check that the condition for semisimplicity here coincides with that in  \cite{rui2022jm}, Theorem A.
	It may be helpful to restate Theorem \ref{qBrauersolution} and \cite{rui2022jm}, Theorem A, in the following form,  cf.
\cite{Rui-BMW-Gram-determinants}, Theorem 5.9.

	Suppose $q \ne \pm 1$ and $r \ne \pm 1$, and let $e$ denote the order of $q^2$
	as a root of unity, with $e = \infty$ in case $q$ is not a root of unity. Then
	the $q$-Brauer algebra $\qBr_n^K(r,q,\delta)$ is semisimple if and only if $n <
		e$ and $$ r \not \in \{\pm q^{4 - k} : n \ge k \ge 5 \} \ \cup \ \{ \pm
		q^{4-2k} : n \ge k \ge 3\} \ \cup \ \{\pm q^{k-2} : n \ge k \ge 3\}. $$
\end{remark}

\begin{example}
	\label{Brauerreference}  Using $m_0$ in the statement of points (3) and (4) in Theorem \ref{Brauersolution} allows  us to avoid listing  a
multitude of  cases depending on the sign of $N$ and the parity of $e$ and $N$.  Still, it may be useful to present some sample cases in order to
form a more concrete idea of the result.
	\begin{enumerate}
		\item  Suppose $q^2$ is a primitive $e$-th root of unity, $q ^e = 1$, and $r=\pm q^N$,
		      where $-e < N < 0$. Necessarily, $e$ is odd. Consider the case that $N$ is odd,
		      $N-e$ is even. Then $$ m = \min\{ e-1, (e-N)/2 + 1, -N + 3, N+e + 1\}. $$

		\item  Suppose $q^2$ is a primitive $e$-th root of unity, $q^e = -1$, and $r=\pm q^N$,
		      where $-e < N < 0$, and $N \ne 0$. Consider the case that $N$ is odd and $e$ is
		      even. Then $$ m = \min\{e-1, -N +3 , N+e + 1\}. $$
	\end{enumerate}
\end{example}

\subsection{BMW algebras} We can give the formulas for $d_\la^F$ in terms of $\bm r$ and $\bm q$ by the
same argument as for the algebras $\qBr_n^F$.
\begin{theorem}
	\label{theorem:  weights in the BMW case}
	{\normalfont (\cite{Wenzl-BCD}, Theorem 5.5, \cite{TW}, (7.3))} The formula for the quantities $d_\la^F$   for the {\normalfont BMW}  algebras
$\BMW_n^F(\bm r,\bm q)$,  with $F = \Q(\bm r, \bm q)$,   coincides with that for the $q$-Brauer algebra $\qBr_n^F(\bm r,\bm q)$ after substituting
$\bm r$ by $\bm r \bm q$,   except that the  factors corresponding to diagonal boxes $(i,i)$ of $\la$ are replaced by

	\begin{equation}
		\label{eqn: diagonal factors for BMW}   \frac{(\bm r\bm q-\bm q^{-2\la_i+2i})(\bm r^{-1}\bm q^{-1}+\bm q^{-2\la_i'+2i-2})}{1-\bm
q^{-2h_\la(i,i )}}\
	\end{equation}
\end{theorem}

\begin{remark}
	For a Young diagram $\la$ and $(i, j) \in \la$,  let $a_\la(i, j) = \la_i + \la_j - i -j$ and  $b_\la(i, j) = -\la'_i - \la'_j  + i + j  -2$.
Thus, $d_\la(i, j) =  a_\la(i, j) $  for $i \le j$ and  $d_\la(i, j) =  b_\la(i, j) $  for $i > j$.    The analogue of Remark \ref{tabmonotone}
holds, namely $b_\la(i, j)$ increases as $i$ or $j$ increases, for $i \ge j$.    With this notation, the diagonal factors in (\ref{eqn: diagonal
factors for BMW})  are
	\begin{equation}
		\label{eqn:  BMW diagonal factors rewritten}
		\frac{(\bm r\bm q-\bm q^{-a_\la(i,i)})(\bm r^{-1}\bm q^{-1}+\bm q^{b_\la(i,i)})}{1-\bm q^{-2h_\la(i,i )}}\
	\end{equation}
	When $r = \eps q^{N-1}$  (in a field $K$)   with $\eps \in \{1,-1\}$,   the diagonal factors evaluate to
	\begin{equation}
		\label{eqn:  BMW diagonal factors rewritten again}
		\frac{( \eps   q^N -  q^{-a_\la(i,i)})(\eps   q^{-N} +  q^{b_\la(i,i)})}{1-  q^{-2h_\la(i,i )}}\
		=  \frac{( 1-  \eps q^{-N -a_\la(i,i)})(1 +  \eps q^{N + b_\la(i,i)})}{1-  q^{-2h_\la(i,i )}}\
	\end{equation}

\end{remark}

To develop specific semisimplicity criteria for the BMW algebras, we need to
find vanishing conditions for the factors in (\ref{eqn: BMW diagonal factors
	rewritten again}).

\begin{lemma}
	\label{lemma:  m 3 of N}
	Let $m_3(N) $ denote the least integer $n \ge 2$ such that there exists a Young diagram $\la$ of size $n$  and $(i, i) \in \la$ with
	$N + b_\la(i,i) = 0$.   Then $m_3(N) = N/2$ if $N$ is positive and even, and $m_3(N) = \infty$ otherwise.
\end{lemma}

\begin{proof}
	Similar to Lemma \ref{lemma:  n zero prime for Brauer}, part (2).
\end{proof}

\begin{remark}
	\label{remark:  BMW r = power of q}
	In the following, we consider the possibility $r = \pm q^a$.  We do not allow $r = q^{-1}$  or $r = -q$, since both possibilities imply $\delta
= 0$, which we exclude from our analysis.    If $K$ is a field in which $q=\pm 1$, $\BMW^K_n(r,q,\delta)\cong \Br^K_n(\delta)$. Hence it suffices
to consider fields $K$ where $q\neq\pm 1$.

	Suppose that $q$ is a root of unity in a field $K$. Let $e$ denote the order of
	$q^2$ and $f$ the order of $q$ as a root of unity. Then $f$ is even if and only
	if $f = 2e$, and in this case $q^e = -1$. If $f$ is odd, then $f = e$, so $q^e
		= 1$.

	If $q$ is a root of unity and $r = \pm q^{N-1}$ for some integer $N$, then
	since $q^e = \pm 1$, for all integers $k$, $r = \pm q^{k e + N -1}$. Replacing
	$N$ by $N + ke$, we can suppose that $-e < N \le 0$.
\end{remark}

\begin{notation}
	Consider the specialization $d_\la^K$  of the quantities from Theorem \ref{theorem:  weights in the BMW case} to parameters $r, q$ in a field
$K$,  where $r = \eps q^{N-1}$ for some integer $N$ and $\eps \in \{1, -1\}$.
	\begin{enumerate}
		\item  Let $m_1'(N, q)$ denote the least $n \ge 2$ such that there exists a Young
		      diagram $\la$ of size $n$ and there exists $(i, j )\in \la$ with $i \ne j$ such
		      that $q^{2(N + d_\la(i,j))} = 1$.
		\item Let $m_2'(N, \eps, q)$ denote the least $n \ge 2$ such that there exists a
		      Young diagram $\la$ of size $n$ and there exists $i$ with $(i,i)\in \la$ such
		      that $1- \eps q^{N + d_\la(i,i)} = 0$.
		\item  Let $m_3'(N, \eps, q)$ denote the least $n \ge 2$ such that there exists a
		      Young diagram $\la$ of size $n$ and there exists $i$ with $(i,i)\in \la$ such
		      that $1 + \eps q^{N + b_\la(i,i)} = 0$.

	\end{enumerate}
\end{notation}

\begin{theorem}
	\label{BMWsolution}    Consider the {\normalfont BMW}  algebras   $\BMW_n^K(r,q,\delta)$ over a field $K$ in which $q \ne \pm 1$.
	\begin{enumerate}
		[label = {\normalfont (\arabic*)}]
		\item  $\BMW_n^K(r,q,\delta)$ is semisimple for all $n\in \N$ if $q$ is not a root of unity and $r\neq \pm q^a$ for any $a\in \Z$.
		\item  If $q^2$ is a primitive $e$--th root of unity and $r\neq \pm q^a$ for any $a\in
			      \Z$, then $\BMW_n^K(r,q,\delta)$ is semisimple if and only if $n < e$.
		\item Suppose $q$ is not a root of unity and $r = \eps q^{N-1}$ with $\eps = \pm 1$.
		      If $\eps = 1$, then $\BMW_n^K(r,q,\delta)$ is semisimple if and only if $n \le
			      m_0(N)$. If $\eps = -1$, then $\BMW_n^K(r,q,\delta)$ is semisimple if and only
		      $n \le \min \{ m_1(N), m_3(N)\}$.
		\item   Suppose that $q$ is a root of unity and $r = \pm q^{N-1}$. If $e$ is the order
		      of $q^2$ as a root of unity we can assume $-e < N \le 0$, and $r = \eps
			      q^{N-1}$ with $\eps = \pm 1$. Then $\BMW_n^K(r,q,\delta)$ is semisimple if and
		      only if $$ n \le \min\{ e-1, m'_1(N, q), m'_2(N, \eps, q), m'_3(N, \eps, q) \}.
		      $$ The following lemma will describe how to compute the quantities $m'_i$ in
		      the various cases that can occur.

	\end{enumerate}
\end{theorem}

\begin{proof}
	The proof of  parts (1) -- (3) is  similar to Theorem \ref{qBrauersolution} parts (1) -- (3).    In particular, one uses that
	$n_1 = \infty$ if $q$ is not a root of unity, and $n_1 = e-1$ if $q^2$ is a primitive $e$-th root of unity,  since the Hecke algebras $H_n(K;
q^2)$  play the role of $Q^K_n$.   It follows from this that the quantities $d_\la^F$  are $K$ evaluable as long as $|\la| \le e-1$.     Hence, in
case (4),
	$$m = \min\{n_0, n_1\}   = \min\{  e-1,   m'_1(N, q),  m'_2(N, \eps, q),  m'_3(N, \eps, q)  \}, $$
	by Theorem  \ref{theorem:  weights in the BMW case} and  equation (\ref{eqn:  BMW diagonal factors rewritten again}) and   
Proposition \ref{traceevaluation}.
\end{proof}

The following lemma describes how to compute the quantities $m_i'$, and thus
$m = \min\{n_0, n_1\}$ when $r = \pm q^{N-1}$ and $q$ is a root of unity. The quantities
$m_i(N)$ that appear in the statement were determined in Remark \ref{remark: m
	sub i of zero}, Lemma \ref{lemma: n zero prime for Brauer}, and Lemma
\ref{lemma: m 3 of N}.

\begin{lemma}
	\label{lemma:  BMW explicit conditions}   Let $K$ be a field,  let $q \ne \pm 1$ be a root of unity in $K$,  and let $e$ denote the order of
$q^2$ and $f$ the order of $q$ as a root of unity.  Suppose $r = \eps q^{N-1}$  where $\eps  = \pm 1$  and $-e < N \le 0$,  see Remark \ref{remark:
BMW r = power of q}.
	\begin{enumerate}
		[label = {\normalfont (\arabic*)}]
		\item  $m'_1(N, q) = \min \{m_1(N), m_1(N-e),  m_1(N+e)\}$.
		\item  Suppose $\character K \ne 2$ and $\eps = 1$. Then $$ m'_2(N, 1, q) =
			      \min\{m_2(N), m_2(N-f)\}, $$ and $$ m'_3(N, 1, q) =
			      \begin{cases}
				      \infty   & \text{if  $f$ is odd}   \\
				      m_3(N+e) & \text{if $f$ is even }. \\
			      \end{cases}
		      $$
		\item  Suppose $\character K \ne 2$ and $\eps = -1$. Then $$ m'_2(N, -1, q) =
			      \begin{cases}
				      \infty   & \text{if  $f$ is odd}  \\
				      m_2(N-e) & \text{if $f$ is even } \\
			      \end{cases}
		      $$
		      and
		      $$
			      m'_3(N, -1, q) =  \min\{m_3(N + f), m_3(N+2f)\}.
		      $$
		\item Suppose $\character K = 2$. Then (since $-1 = 1$) $$ m'_2(N, 1, q) =
			      \min\{m_2(N), m_2(N-f)\}, \text{and} $$ $$ m'_3(N, 1, q) = \min\{m_3(N + f),
			      m_3(N+2f)\}. $$
	\end{enumerate}
\end{lemma}

\begin{proof}
	This is all straightforward.  We prove (3)  to indicate the pattern.  For $m'_2$,  we have to examine solutions to
	$
		1 + q^{N + d_\la(i, i)} = 0.
	$
	If  $f $ is odd,  this has no solutions when $\character K \ne 2$.    If $f = 2e$ is even,   then it is equivalent to $N + d_\la(i, i) = (2k +
1) e$, or
	$N - (2k +1)e + d_\la(i, i) = 0$.   Hence $m'_2(N, -1, q)$ is the minimum over $k$ of $m_2(N- (2k + 1)e)$.  However,  values of $k$ other than
$k =0$ produce redundant bounds, and hence $m'_2(N, -1, q) =m_2(N-e)$.

	For $m_3'$, we have to look at solutions to $1 - q^{N + b_\la(i, i)}$. This is
	equivalent to $N + b_\la(i, i) = k f$, or $N - k f + b_\la(i, i) = 0$. Hence,
	$m'_3(N, -1, q)$ is the minimum over $k$ of $m_3(N - k f)$. Values of $k$ other
	than $k = -1, -2$ give redundant bounds, so $m'_3(N, -1, q) =\min\{m_3(N + f),
		m_3(N+2f)\}$.
\end{proof}

\begin{remark}
	One can check that the semisimplicity criterion of Theorem \ref{BMWsolution}  and Lemma \ref{lemma:  BMW explicit conditions}  coincides with
that of  \cite{Rui-BMW-Gram-determinants}, Theorem 5.9.
\end{remark}

\begin{example}
	Consider the {\normalfont BMW}  algebras   $\BMW_n^K(r,q,\delta)$ over a field $K$ in which $q \ne \pm 1$.  Suppose $q$ is a root of unity and
$\ord(q^2) = e$,   $\ord(q) = 2e$.
	Suppose $r = -q^{N-1}$  with $-e < N \le 0$.   Suppose in addition that $N$ is even and $e$ is odd.   Verify that
	$$
		\begin{aligned}
			m'_1(N, q)     & =  \min\{ N+ e +1, -N + 3\} \\
			m'_2(N, -1, q) & =  \infty                   \\
			m'_3(N, -1, q) & = (N+f)/2 = N/2 + e.
		\end{aligned}
	$$
	Thus,  $\BMW_n^K(r,q,\delta)$  is semisimple if and only if
	$$
		n \le \min \{ e-1, N + e + 1, -N +3,  N/2 + e\}.
	$$
\end{example}

\appendix
\section{Some remarks on the $q$-Brauer algebras} \label{appendix}
We provide some observations on the $q$-Brauer algebras that are needed in the
main text.

\subsection{The involution} \label{appendix: involution}
\begin{lemma}
	\label{Lemma additional e 2 relations} \phantom{xxx}
	\begin{enumerate}
		[label = {\normalfont (\arabic*)}]
		\item  {\normalfont (\cite{Wenzl-qBr1}, Lemma 3.2) }
		      $
			      \ex 2 = e g_2\inv  g_3 g_1\inv  g_2  e = e g_2\inv g_3\inv  g_1 g_2 e = e g_2 g_3\inv g_1 g_2\inv.
		      $
		\item  $g_3 \powerpm \ex 2 = q \powerpm \ex 2 = \ex 2 g_3 \powerpm$.
		\item   $g_2\inv  g_3 g_1\inv  g_2 \ex 2 = \ex 2 = g_2\inv g_3\inv  g_1 g_2  \ex 2$.
		\item  $ \ex 2  g_2\inv  g_3 g_1\inv  g_2 = \ex 2 =  \ex 2 g_2\inv g_3\inv  g_1 g_2 $.
	\end{enumerate}
\end{lemma}

\begin{proof}
	Part (1) is from  \cite{Wenzl-qBr1}, Lemma 3.2.    For part (2),  note that the braid relations give $g_3\powerpm (g_2 g_3 g_1\inv g_2\inv) =
(g_2 g_3 g_1\inv g_2\inv)  g_1\powerpm$, so the assertion follows from Definition \ref{def qBr}  parts (3) and (4).
	For the first equality in part (3),
	$$
		\begin{aligned}
			g_2\inv  g_3 g_1\inv  g_2 \ex 2 & =    (g_2 - (q - q\inv) )  g_3 g_1\inv  (g_2 \inv + (q - q\inv) )  \ex 2
\\
			                                & = g_2  g_3 g_1\inv g_2 \inv \ex 2 + (q - q\inv) \left(  - g_3 g_1\inv g_2 \inv + g_2   g_3 g_1\inv
-  (q - q\inv)    g_3 g_1\inv                                            \right ) \ex 2 .
		\end{aligned}
	$$
	Note that $g_2  g_3 g_1\inv g_2 \inv \ex 2  = \ex 2$  and $g_3 g_1\inv g_2 \inv  \ex 2  = g_2 \inv \ex 2$,
	by Definition \ref{def qBr} part (4).   We have $g_3 g_1 \inv \ex 2 = \ex 2$  by  Definition \ref{def qBr} part (3) and statement (2) of this
Lemma. So the last expression simplifies to
	$$
		\ex 2 + (q - q\inv) \left(  -g_2\inv +  g_2 - (q - q\inv)    \right)  \ex 2   = \ex 2,
	$$
	since $$ -g_2\inv +  g_2 - (q - q\inv)  = 0.$$
	To get the second equality in part (3),  multiply both sides of the first equality by $g_2\inv g_3\inv  g_1 g_2 $ on the left.  The proof of
part (4) is similar.
\end{proof}

In the following statement, ``algebra involution" means an algebra
anti-automorphism of order 2.

\begin{proposition}
	{\normalfont (Compare ~\cite{Nguyen1}, Proposition 3.14)}  There is a unique algebra involution $*$ on $\qBr_n$  that fixes the generators,
i.e. $e^* = e$, $g_i^* = g_i$ and $(g_i \inv)^* = g_i\inv$.
\end{proposition}

\begin{proof}
	Let $\mathcal F$ be the free associative unital algebra on the symbols $e$ and $g_1\powerpm, \dots, g_{n-1}\powerpm$, and let $\mathcal I$ be
the ideal in $\mathcal F$ generated by the defining relations for $\qBr_n$ and the relations $g_i g_i\inv = g_i\inv g_i = 1$.    Let $*$ be the
algebra involution on $\mathcal F$ fixing the generators.
	The additional relations given in the previous Lemma show that $\mathcal I$ has a $*$-invariant set of generators and therefore  is
$*$--invariant.
\end{proof}

\subsection{Temperley-Lieb elements}   \label{appendix TL}
Recall the elements $e_j$ from Definition \ref{TL elements in q-Brauer}, and
the subalgebras $\qBr_{j, n} = \alg\{e_j, g_j, g_{j+1}, \dots, g_{n-1}\}$,

\begin{lemma}
	\label{commutation of e j and g k}    $e_j$ commutes with $g_k$ for $k \ge j+2$.
\end{lemma}

\begin{proof}
	For $j = 1$, this is in Definition \ref{def qBr}.  For $j >1$, it follows from the definition of $e_j$ and induction on $j$.
\end{proof}

\begin{lemma}
	\label{lemma: B4 isomorphisms}  For $j \ge 1$ and $k \ge 2$,    $e_1 \mapsto e_{j+1}$  and $g_\ell \mapsto g_{\ell +j }$  for $1 \le \ell \le
k-1$ determines an isomorphism $\psi_{j}$ from
	$\qBr_{k}$ to $\qBr_{j+1, j+k}$
\end{lemma}

\begin{proof}
	Fix any $m \ge j+k -1$.  For any $s$ with $1 \le s \le j$,  since $e_s$ commutes with $g_t$ for $t \ge s+2$,
	$$
		e_{s+1} =
		\begin{cases}
			\Ad(g_{s, m}^+ )(e_s) & \text{if} \   s \ \text{is odd}   \\
			\Ad(g_{s, m}^- )(e_s) & \text{if}  \  s  \ \text{is even} \\
		\end{cases}
		,
	$$
	while $\Ad(g_{s, m}^\pm )(g_\ell)  = g_{\ell+1}$ for $s \le \ell \le m-1$.
	By induction on $j$,
	$$
		\Ad(g_{j, m}^\pm \cdots  g_{3, m}^+  g_{2, m}^-  g_{1, m}^+)
	$$
	maps $e_1 \mapsto e_{j+1}$ and $g_\ell \mapsto g_{\ell+j }$  for $1 \le \ell \le k-1$, with the first factor on the left being $g_{j, m}^+$ if
$j$ is odd and  $g_{j, m}^-$ if $j$ is even.
\end{proof}

We obtain a number of identities involving the elements $e_j$.

\begin{lemma}
	\label{braid like relations on e's and g's}
	For all admissible values of $j$,
	\begin{enumerate}
		[label = {\normalfont (\arabic*)}]
		\item $e_j^2 = \delta e_j$.
		\item  $g_j\powerpm e_j = e_j g_j\powerpm = q\powerpm e_j$.
		\item $e_j  g_{j+1}\powerpm e_j = r\powerpm e_j$.
		\item $e_{j+1} g_j\powerpm e_{j+1} = r\powerpm e_{j+1}$.
		\item $e_j e_{j+1} e_j = e_j$ and $e_{j+1} e_j e_{j+1} = e_{j+1}$.
		\item  $$e_j e_{j+1} = q r e_j g_{j+1}\inv g_j \inv = q r  g_{j+1}\inv g_j \inv e_{j+1}$$ if $j$ is odd, and
		      $$ e_j e_{j+1} = q\inv r\inv e_j g_{j+1}g_j =    q\inv r\inv g_{j+1}g_j  e_{j+1}$$
		      if $j$ is even, and similarly for $e_{j+1} e_j$ (with the order of all words reversed).
		\item $e_j e_{j+1} g_j$ and $e_j g_{j+1} g_j$  are in the span of $e_j$, $e_j e_{j+1}$, and $e_j g_{j+1}$;  and similarly for $g_j
e_{j+1} e_j$ and $g_j g_{j+1} e_j$  (with the order of all words reversed).
	\end{enumerate}
\end{lemma}

\begin{proof}
	Parts (1), (2), and (3) are immediate from  Definition \ref{def qBr} and  Lemma \ref{lemma: B4 isomorphisms}.   For parts (4),  (5), and (6)
write each instance of $e_{j+1}$ as
	$\Ad(g_j g_{j+1})(e_j)$  or $\Ad(g_j\inv g_{j+1}\inv)(e_j)$, depending on the parity of $j$.    The results then follow from straightforward
computations, using parts (2) and (3) and braid  relations.  For part (7), use part (6) and the quadratic relation for the braid generators.
\end{proof}

Let $G_i= g_i g_{i+1} g_{i-1}\inv g_i\inv$. Let $\sigma_i$, $i \ge 1$, denote
the generators of the Artin braid group of type $A$, satisfying $\sigma_i
	\sigma_j = \sigma_j \sigma_i$ if $|i-j| \ge 2$, and $\sigma_i \sigma_{i+1}
	\sigma_i = \sigma_{i+1} \sigma_i \sigma_{i+1}$ for $i \ge 1$. Then $\sigma_i
	\mapsto G_{2i}$ and $\sigma_i \mapsto G_{2i-1}$ are braid group homomorphisms;
cf. \cite{Wenzl-qBr1}, Lemma 1.3. Let $G_i' = g_i\inv g_{i+1}\inv g_{i-1} g_i =
	(G_i\inv)^* $. With this notation, we have $$ e_{(2)} = e G_2 e = e G_2\inv e =
	e G_2^* e = e G_2' e. $$ Moreover, $\ex 2$ is invariant under multiplication on
the left or right by $G_2$, $G_2\inv$, $G_2^*$, or $G_2'$, see Lemma \ref{Lemma
	additional e 2 relations}.

For $k \ge 1$, write $e_{(k, 2)} = e_k G_{k+1} e_k$. The isomorphism $\qBr_4
	\to \alg\{e_k, g_k, g_{k+1}, g_{k+2}\}$, mapping $e \to e_k$ and $g_1, g_2,
	g_3$ to $g_k, g_{k+1}, g_{k+2}$, respectively, also maps $G_2 \mapsto G_{k+1}$,
and $e_{( 2)} \mapsto e_{(k, 2)} $. Hence, $$ e_{(k, 2)} = e_k G_{k+1} e_k =
	e_k G_{k+1}\inv e_k = e_k G_{k+1}^* e_k =e_k G_{k+1}' e_k.$$ Moreover, $\ex {k,
		2}$ is invariant under multiplication on the left or right by $G_{k+1}$,
$G_{k+1}\inv$, $G_{k+1}^*$, or $G_{k+1}'$. Depending on the parity of $k$,
$e_{k+2}$ is either $\Ad(G_{k+1})(e_k)$ or $\Ad(G_{k+1}')(e_k)$. In the first
case,
\begin{equation}
	\label{e k 2}
	e_k e_{k+2} = e_k G_{k+1} e_k G_{k+1}\inv= e_{(k, 2)}G_{k+1}\inv= e_{(k, 2)} = e_k G_{k+1} e_k.
\end{equation}
and similarly in the second case, $e_k e_{k+2} =  e_{(k, 2)}$.  Likewise, in either case,
$
	e_{k+2} e_k = e_{(k, 2)},
$
so $e_k$ commutes with $e_{k+2}$.

\begin{lemma}
	\label{lemma: commutation of ej}  $e_j$ commutes with $e_k$   if $|j-k| \ge 2$.
\end{lemma}

\begin{proof}
	We have just seen that $e_k$ commutes with $e_{k+2}$.   Let $j \ge k+2$.  Then $e_{j+1} = \Ad(g_j g_{j+1})(e_j)$  or
	$e_{j+1} = \Ad(g_j\inv g_{j+1}\inv)(e_j)$, depending on the parity of $j$. In either case, if $e_k$ commutes with $e_j$,  then $e_k$ commutes
with $e_{j+1}$.  By induction, $e_k$ commutes with $e_j$ for all $j \ge k+2$.
\end{proof}

\begin{remark}
	Lemma \ref{braid like relations on e's and g's} parts (1) and (5) and Lemma \ref{lemma: commutation of ej}  show that the elements $e_j$
satisfy the Temperley-Lieb relations.
\end{remark}

\subsection{The algebras $A_n$ and the Temperley-Lieb elements}  \label{appendix An properties}

As in the main text, fix some large integer $M$ and define, for $n \ge 1$ $$A_n
	= \qBr_{M - n + 1, M} = \alg\{ e_{M-n+1}, g_{M-n+1}, \dots, g_{M-1}\} =
	\psi_{M-n}(\qBr_n). $$ and $A_0 = S$.

\begin{lemma}
	\label{A k and A k-1} \phantom{xxx}
	\begin{enumerate}
		[label = {\normalfont (\arabic*)}]
		\item For each $k \ge 2$, $A_k$ is the span of $A_{k-1}$, $A_{k-1} e_{M-k+1}
			      A_{k-1}$, and $A_{k-1} g_{M-k+1} A_{k-1}$,
		\item For each $k \ge 2$, $A_k e_{M-k+1} = A_{k-1} e_{M-k+1} $.
	\end{enumerate}
\end{lemma}

\begin{proof}
	For part (1), $A_2$ is the span of $\bm 1$,  $e_{M-1}$, $g_{M-1}$,  so the assertion is valid for $k = 2$.  Assume the assertion holds for some
value of $k \ge 2$ and consider $A_{k+1}$, which is the algebra generated by $A_k$ and by $e_{M-k}$, $g_{M-k}$.  Consider a word  $w$ in
	$e_{M-k}$, $g_{M-k}$, and elements of $A_k$, with two or more instances of $e_{M-k}$ or  $g_{M-k}$.  We need to show that $w$ is in the span of
words with fewer instances of  $e_{M-k}$ or  $g_{M-k}$.  Consider a subword $w_0 =\chi a \chi'$, where $a \in A_k$ and
	$\chi, \chi' \in \{e_{M-k}, g_{M-k}\}$.   Applying the induction hypothesis to $a$, either $a \in A_{k-1}$, or $a$ is in the span of words
	$a_1 \eta a_2$, with $a_i \in A_{k-1}$, and $\eta \in \{ e_{M-k+1}, g_{M-k+1}    \}$.   In the former case, $w_0 = a \chi \chi'$, and in the
latter case, $w_0$ is in the span of words $a_1 \chi \eta \chi' a_2$.  We now refer to Lemma \ref{braid like relations on e's and g's}, which
(along with the quadratic relation on the braid generators) shows that any such word is in the span of words with at most one instance  $e_{M-k}$
or  $g_{M-k}$.   The result follows by induction on the number of instances of  $e_{M-k}$ or  $g_{M-k}$ in the original word $w$.

	Part (2) follows by exactly the same argument. One only has to observe that
	(using Lemma \ref{braid like relations on e's and g's}) for $\chi\in \{e_{M-k},
		g_{M-k}\}$ and $\eta \in \{ e_{M-k+1}, g_{M-k+1} \}$, the words $\chi e_{M-k}$
	and $\chi \eta e_{M-k}$ are in $A_k e_{M-k}$.
\end{proof}

\begin{lemma}
	\label{Proposition  Ak properties}    \phantom{xxxx}
	\begin{enumerate}
		[label = {\normalfont (\arabic*)}]
		\item  $A_k e_{M-k+1} A_k = A_{k-1} e_{M-k+1} A_{k-1}$.
		\item  $e_{M-k+1} \in A_k$ commutes with $A_{k-2}$.
		\item $e_{M-k+1} A_k e_{M-k+1} = e_{M-k+1} A_{k-2}$
		\item $x \mapsto x \bar e_{M-k+1}$ is  injective  from $A_{k-1}$ to $A_k$.
		\item  In particular, $x \bar e_{M-k+1}$ is an injective homomorphism from $A_{k-1}$
		      to $A_k$.
	\end{enumerate}
\end{lemma}

\begin{proof}
	Point (1) is immediate from Lemma \ref{A k and A k-1}, and (2) follows from Lemmas \ref{commutation of e j and g k}  and \ref{lemma:
commutation of ej}.   For (3),  $e_{M-k+1} A_k e_{M-k+1} = e_{M-k+1} A_{k-1} e_{M-k+1}$ from Lemma \ref{A k and A k-1}.  But
	$A_{k-1}$ is spanned by $A_{k-2}$, $A_{k-2} e_{M-k+2} A_{k-2}$ and  $A_{k-2} g_{M-k+2} A_{k-2}$.  So (3) follows from (2) and Lemma \ref{braid
like relations on e's and g's}.

	For point (4), consider the algebras over the generic ground ring $R$, and
	recall that the specialization $\qBr_M(\Z[\bm \delta\powerpm])$ is isomorphic
	to the Brauer algebra over $\Z[\bm \delta\powerpm]$; moreover, because of the
	freeness of $\qBr_M$ over $R$, the specialization map $\theta: x \mapsto x
		\otimes_{R_0} \Z[\bm \delta\powerpm]$ is injective. The specialization maps
	$e_j$ in the $q$-Brauer algebra to $e_j$ in the Brauer algebra and $g_j$ to
	$s_j$. Hence, it maps $A_k = \alg\{e_{M-k+1}, g_{M-k+1}, \dots, g_{M-1}\}$ to
	$B_{M-k+1, n}$, the Brauer subalgebra on strands $M-k+1$ to $n$. The
	specialization intertwines the maps $x \mapsto e_{M-k+1} x$ from $A_{k-1}$ to
	$A_k$ with the map from $B_{M-k+2, M}$ to $B_{M-k+1, M}$ determined on the
	level of Brauer diagrams by $d \mapsto e_{M-k+1} \otimes d$. Therefore the
	injectivity of $x \mapsto e_{M-k+1} x$ (for the algebras defined over $R$)
	follows from the injectivity of $\theta$. Finally, injectivity for the algebras
	defined over an arbitrary ground ring $S$ follows from that over $R$. Point (5)
	is immediate from point (4).
\end{proof}

\begin{remark}
	Lemmas \ref{A k and A k-1} and \ref{Proposition  Ak properties}  verify axioms  (\ref{axiom: en An en}) and  (\ref{axiom:  An en}) for the
$q$-Brauer algebras.
\end{remark}

\subsection{Nguyen's cellularity theorem}   \label{appendix subsection: Nguyen}
We give an efficient proof of Nguyen's cellularity theorem for the $q$-Brauer
algebras, Theorem \ref{Nguyen's theorem}.

In this section $\qBr_n$ denotes the $q$-Brauer algebra over an arbitrary ring
$S$ and $H_n$ the Hecke subalgebra. According to ~\cite{Wenzl-qBr1},
Proposition 3.5, $\qBr_n = \sum_{m \ge 0} H_n \ex s H_n$, and $I(n, s) :=
	\sum_{m \ge s} H_n \ex s H_n$ is an ideal in $\qBr_n$. One easily checks that
$I(n, s)$ is the ideal generated by $\ex s$. Evidently $I(n, s+1) \subseteq
	I(n, s)$.

\begin{lemma}
	\label{ex s absorbs K s}  Let $s \le n/2$.
	\begin{enumerate}
		[label = {\normalfont (\arabic*)}]
		\item  For $1 \le i \le s-1$, $$ G_{2i} \ex s = \ex s G_{2i} = \ex s, $$ where $G_{2i}
			      = g_{2i} g_{2i+1} g_{2i-1}\inv g_{2i}\inv$.
		\item  For $1 \le i \le s$, $$ g_{2i-1} \ex s = \ex s g_{2i-1} = q \ex s. $$
	\end{enumerate}
\end{lemma}

\begin{proof}
	This follows from Lemma \ref{product expression for ex s}, the computation preceding Lemma \ref{lemma: commutation of ej} , and Lemma
\ref{braid like relations on e's and g's}, part (2).
\end{proof}

Recall the family of permutations $\mathcal D'_{n, s}$ from Definition
\ref{definition: shortest coset reps q-Brauer}. Moreover, define $K_s$ to be
the subgroup of $\mathfrak S_{2s}$ generated by $s_1$ and the elements $w_i =
	s_{2i} s_{2i +1} s_{2i-1} s_{2i}$ for $1 \le i \le s-1$. Note that $K_s$ is
isomorphic to the semidirect product of $(\Z_2)^s$ and $\mathfrak S_s$, where
$\langle s_{2i -1} : 1 \le i \le s \rangle \cong (\Z_2)^s$ and $\langle w_i : 1
	\le i \le s-1\rangle \cong \mathfrak S_s $, acting by permutations of the set
of pairs $ \{ \{1, 2\}, \{3, 4\}, \dots, \{2s-1, 2s\}\}. $ If $\pi = \pi_1
	\pi_2$, where $\pi_1 \in \langle s_{2i -1} : 1 \le i \le s \rangle$ and $\pi_2
	\in \langle w_i : 1 \le i \le s-1\rangle$, then $\ell(\pi) = \ell(\pi_1) +
	\ell(\pi_2)$.

\begin{lemma}
	{\normalfont (\cite{Enyang-cellular-bases},  Proposition 3.1) } \label{enyang lemma} $\mathcal D'_{n, s}$ is a complete set of left coset
representatives of  $K_s \times \sym_{2s+1, n}$ in $\sym_n$.  If $u \in \mathcal D'_{n, s}$ and $w \in K_s \times \sym_{2s+1, n}$, then $\ell(u w)
= \ell(u) + \ell(w)$,  where $\ell$ denotes the usual length function in the symmetric group.
\end{lemma}

\begin{remark}
	\label{remark: factorization of Brauer diagrams}   Every Brauer diagram $d$ on $n$ strands is either a permutation diagram, or, if $d$ has $n -
2s$ through strands for $1 \le s \le n/2$,  then $d$
	can be written uniquely as
	$$
		d =  u e_1 e_3 \cdots e_{2s-1} \pi v\inv,
	$$
	as a product in the Brauer algebra $\Br_n$,   where $u, v \in \mathcal D'_{n,s}$ and  $\pi \in \sym_{2s+1, n}$.
\end{remark}

\begin{lemma}
	\label{H n e s}   Let $1 \le s \le n/2$.
	$$
		H_n \ex s = \spn\{g_u \ex s H_{2s+1, n} : u \in \mathcal D'_{n,s}\}.
	$$
\end{lemma}

\begin{proof}
	We will use the notation $\overline G_{2i} = g_{2i} g_{2i + 1} g_{2i -1} g_{2i}$ in the proof.
	Let $\sigma  \in \sym_n$.  We prove by induction on $k = \ell(\sigma)$ that
	$$
		g_\sigma \ex s \in  \spn\{g_u \ex s H_{2s+1, n} : u \in \mathcal D'_{n,s}\},
	$$
	the assertion being evident for $k = 0$.  Assume the assertion holds for shorter elements.   By Lemma \ref{enyang lemma}, $\sigma = u \pi$ with
	$ u \in \mathcal D'_{n, s}$ and  $\pi \in K_s \times \sym_{2s+1, n}$, and $k = \ell(u) + \ell(\pi)$.  Moreover,  $\pi = \pi_1 \pi_2 \pi_3$,
where
	$\pi_1 \in \langle s_{2i} s_{2i+1} s_{2i-1} s_{2i} : 1 \le i \le s-1\rangle $,
	$\pi_2 \in \langle s_{2i -1} : 1 \le i \le s\rangle $, and $\pi_3 \in \sym_{2s+1, n}$;  and $\ell(\pi) = \ell(\pi_1)  + \ell(\pi_2)  +
\ell(\pi_3)$.   Hence
	$$
		g_\sigma \ex s =  g_u g_{\pi_1} g_{\pi_2} \ex s g_{\pi_3} = q^{\ell(\pi_2)}   g_u g_{\pi_1} \ex s g_{\pi_3},
	$$
	using Lemma \ref{ex s absorbs K s}.   If either $\pi_3$ or $\pi_2$ is non trivial, then the conclusion holds from the induction hypothesis
applied to $u \pi_1$.   So we may assume that $\pi_2$ and $\pi_3$ are equal to the identity permutation. Now $g_{\pi_1}$ is a word in $\{\overline
G_{2i}\}$.   Using the quadratic relation in $H_n$, $g_{\pi_1}$ is congruent to the corresponding word in $\{G_{2i}\}$ modulo the span of shorter
reduced words in the generators of $H_n$.  Since $G_{2i} \ex s = \ex s$ by Lemma \ref{ex s absorbs K s}, the conclusion follows from the induction
hypothesis.
\end{proof}

\begin{corollary}
	\label{basis of H e H}  Let $1 \le s \le n/2$.
	$$
		H_n \ex s  H_n = \spn\{g_u \ex s H_{2s+1, n} g_v^*: u, v \in \mathcal D'_{n,s}\}.
	$$
\end{corollary}

{\it Proof of Theorem \ref{Nguyen's theorem}.}  \ \ It follows from Corollary \ref{basis of H e H} that the proposed basis spans $\qBr_n$.
By Remark \ref{remark: factorization of Brauer diagrams},  the cardinality of the proposed basis is the number of Brauer diagrams on $n$ strands.
Since $\qBr_n$ is a free  module over the ground ring with rank the number of Brauer diagrams, it follows that the spanning sets are also linearly
independent.  To check property (2) in Definition \ref{gl cell}, it is enough to check multiplication on the left by a generator of $\qBr_n$.   The
property for multiplication on the left by a Hecke generator $g_i$ follows from Lemma \ref{H n e s} and property (2) for the cellular bases of the
Hecke algebras.   Note that if $(\lambda, n) \in \tilde \Lambda_n$, where $\la$ has size $k = n -2s$,  then $I(n, s+1) \subseteq \qBr_n^{\rhd (\la,
n)}$.  Taking this into account, property (2) for
multiplication on the left by $e$ follows from  Lemma 3.4  in  ~\cite{Wenzl-qBr1} and property (2) for left multiplication by Hecke elements.
Finally, property (3) follows from $\ex s^* = \ex s$ and from the corresponding property for the cellular bases of the Hecke algebras.

\bibliographystyle{amsplain}
\bibliography{semisimplicity}

\end{document}